\documentclass[12pt]{article}

\usepackage[a4paper]{geometry}
\usepackage{amsthm,amssymb}
\usepackage{amsmath,graphicx}
\usepackage[colorlinks=true,citecolor=black,linkcolor=black,urlcolor=blue]{hyperref}
\usepackage{bbm}
\usepackage{ifthen}

\usepackage{ adjustbox }

\usepackage{ tikz }
\usetikzlibrary{ calc }

\theoremstyle{plain}
\newtheorem{theorem}{Theorem}
\newtheorem{lemma}[theorem]{Lemma}
\newtheorem{corollary}[theorem]{Corollary}
\newtheorem{proposition}[theorem]{Proposition}

\theoremstyle{definition}

\theoremstyle{remark}

\DeclareMathOperator{\lcm}{lcm}
\DeclareMathOperator{\scale}{scale}

\def\tn{\textnormal}
\newcommand{\ignore}[1]{}
\newcommand{\set}[1]{\left\{#1\right\}}
\newcommand{\res}[2]{\langle #1, #2 \rangle}
\def\A{\mathcal{A}}
\def\B{\mathcal{B}}
\def\C{\mathcal{C}}

\def\G{\mathcal{G}}
\def\Hc{\mathcal{H}}
\def\K{\mathcal{K}}

\def\S{\mathcal{S}}
\def\T{\mathcal{T}}

\def\W{\mathcal{W}}
\def\mN{\mathbb{N}}
\def\mP{\mathbb{P}}
\def\mR{\mathbb{R}}
\def\mZ{\mathbb{Z}}
\def\a{\alpha}

\def\l{\ell}

\title{Decompositions of Unit Hypercubes \\and the Reversion of a Generalized M\"obius Series}

\author{Yu Hin (Gary) Au\\
\small Department of Mathematics and Statistics\\[-0.8ex]
\small University of Saskatchewan\\[-0.8ex] 
\small Saskatoon, SK, Canada\\
\small\tt au@math.usask.ca}

\begin{document}

\maketitle

\begin{abstract}
Let $s_d(n)$ be the number of distinct decompositions of the $d$-dimensional hypercube with $n$ rectangular regions that can be obtained via a sequence of splitting operations. We prove that the generating series $y = \sum_{n \geq 1} s_d(n)x^n$ satisfies the functional equation $x = \sum_{n\geq 1} \mu_d(n)y^n$, where $\mu_d(n)$ is the $d$-fold Dirichlet convolution of the M\"obius function. This generalizes a recent result by Goulden et al., and shows that $s_1(n)$ also gives the number of natural exact covering systems of $\mZ$ with $n$ residual classes. We also prove an asymptotic formula for $s_d(n)$ and describe a bijection between $1$-dimensional decompositions and natural exact covering systems.
\end{abstract}

\section{Introduction}\label{Sec:01}

\subsection{Decomposing the $d$-dimensional unit hypercube}

Suppose we start with $(0,1)^d$, the $d$-dimensional unit hypercube, and iteratively perform the following operation:
\begin{itemize}
\item
choose a region in the current decomposition, a coordinate $i \in \set{1, \ldots, d}$, and an arity $p \geq 2$;
\item
partition the selected region into $p$ equal smaller regions with cuts orthogonal to the $i^{\tn{th}}$ axis.
\end{itemize}

\begin{figure}[h!]
\begin{center}
\begin{tabular}{ccccccccc}
\begin{tikzpicture}
[scale=0.12, thick]
\draw(0,0) -- (12,0) -- (12,12) -- (0,12) -- (0,0);
\end{tikzpicture}
& \raisebox{0.6cm}{$\rightarrow$~}
\begin{tikzpicture}
[scale=0.12, thick]
\draw(0,0) -- (12,0) -- (12,12) -- (0,12) -- (0,0);
\draw(4,0) -- (4,12);
\draw(8,0) -- (8,12);
\end{tikzpicture}
& \raisebox{0.6cm}{$\rightarrow$~}
\begin{tikzpicture}
[scale=0.12, thick]
\draw(0,0) -- (12,0) -- (12,12) -- (0,12) -- (0,0);
\draw(4,0) -- (4,12);
\draw(8,0) -- (8,12);
\draw(4,6) -- (8,6);
\draw(4,3) -- (8,3);
\draw(4,9) -- (8,9);
\end{tikzpicture}
& \raisebox{0.6cm}{$\rightarrow$~}
\begin{tikzpicture}
[scale=0.12, thick]
\draw(0,0) -- (12,0) -- (12,12) -- (0,12) -- (0,0);
\draw(4,0) -- (4,12);
\draw(8,0) -- (8,12);
\draw(4,6) -- (8,6);
\draw(4,3) -- (8,3);
\draw(4,9) -- (8,9);
\draw(12,6) -- (8,6);
\end{tikzpicture}
& \raisebox{0.6cm}{$\rightarrow$~}
\begin{tikzpicture}
[scale=0.12, thick]
\draw(0,0) -- (12,0) -- (12,12) -- (0,12) -- (0,0);
\draw(4,0) -- (4,12);
\draw(8,0) -- (8,12);
\draw(4,6) -- (8,6);
\draw(4,3) -- (8,3);
\draw(4,9) -- (8,9);
\draw(12,6) -- (8,6);
\draw(10,0) -- (10,6);
\end{tikzpicture}
\end{tabular}
\caption{One way to partition the unit square $(0,1)^2$ into 8 regions.}\label{Fig:0101}
\end{center}
\end{figure}
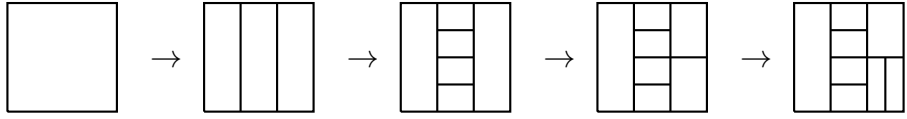

For example, Figure~\ref{Fig:0101} illustrates one way to decompose the unit square $(0,1)^2$ into $8$ regions using a sequence of $4$ partitions. 

We are interested in the following question: Given integers $d,n \geq 1$, how many distinct compositions of $(0,1)^d$ are there with $n$ regions? In Figure~\ref{Fig:0102}, we list all the possible decompositions of $(0,1)^2$ with $n \leq 4$ regions.

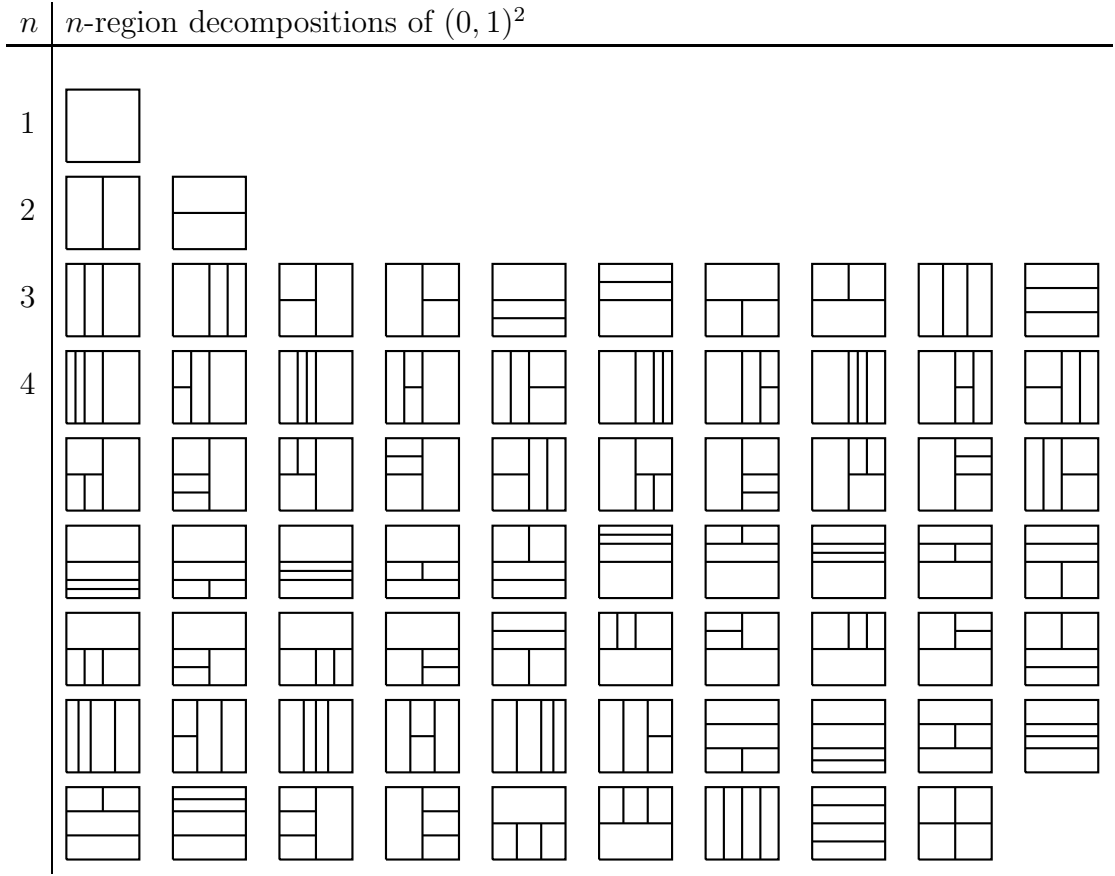
\begin{figure}[h!]
\begin{center}
$
\begin{array}{l|l}
n & \tn{$n$-region decompositions of $(0,1)^2$}\\
\hline
\\
\raisebox{0.4cm}{$1$}
&
\begin{tikzpicture}
[scale=0.08, thick]
\draw(0,0) -- (12,0) -- (12,12) -- (0,12) -- (0,0);
\end{tikzpicture}
\\
\raisebox{0.4cm}{$2$}
&

\begin{tikzpicture}
[scale=0.08, thick]
\draw(0,0) -- (12,0) -- (12,12) -- (0,12) -- (0,0);
\draw(6,0) -- (6,12);
\end{tikzpicture}
\quad
\begin{tikzpicture}
[scale=0.08, thick]
\draw(0,0) -- (12,0) -- (12,12) -- (0,12) -- (0,0);
\draw(0,6) -- (12,6);
\end{tikzpicture}
\\
\raisebox{0.4cm}{$3$}
&

\begin{tikzpicture}
[scale=0.08, thick]
\draw(0,0) -- (12,0) -- (12,12) -- (0,12) -- (0,0);
\draw(6,0) -- (6,12);
\draw(3,0) -- (3,12);
\end{tikzpicture}
\quad
\begin{tikzpicture}
[scale=0.08, thick]
\draw(0,0) -- (12,0) -- (12,12) -- (0,12) -- (0,0);
\draw(6,0) -- (6,12);
\draw(9,0) -- (9,12);
\end{tikzpicture}
\quad
\begin{tikzpicture}
[scale=0.08, thick]
\draw(0,0) -- (12,0) -- (12,12) -- (0,12) -- (0,0);
\draw(6,0) -- (6,12);
\draw(0,6) -- (6,6);
\end{tikzpicture}
\quad
\begin{tikzpicture}
[scale=0.08, thick]
\draw(0,0) -- (12,0) -- (12,12) -- (0,12) -- (0,0);
\draw(6,0) -- (6,12);
\draw(12,6) -- (6,6);
\end{tikzpicture}
\quad
\begin{tikzpicture}
[scale=0.08, thick]
\draw(0,0) -- (12,0) -- (12,12) -- (0,12) -- (0,0);
\draw(0,6) -- (12,6);
\draw(0,3) -- (12,3);
\end{tikzpicture}
\quad
\begin{tikzpicture}
[scale=0.08, thick]
\draw(0,0) -- (12,0) -- (12,12) -- (0,12) -- (0,0);
\draw(0,6) -- (12,6);
\draw(0,9) -- (12,9);
\end{tikzpicture}
\quad
\begin{tikzpicture}
[scale=0.08, thick]
\draw(0,0) -- (12,0) -- (12,12) -- (0,12) -- (0,0);
\draw(0,6) -- (12,6);
\draw(6,0) -- (6,6);
\end{tikzpicture}
\quad
\begin{tikzpicture}
[scale=0.08, thick]
\draw(0,0) -- (12,0) -- (12,12) -- (0,12) -- (0,0);
\draw(0,6) -- (12,6);
\draw(6,12) -- (6,6);
\end{tikzpicture}
\quad
\begin{tikzpicture}
[scale=0.08, thick]
\draw(0,0) -- (12,0) -- (12,12) -- (0,12) -- (0,0);
\draw(4,0) -- (4,12);
\draw(8,0) -- (8,12);
\end{tikzpicture}
\quad
\begin{tikzpicture}
[scale=0.08, thick]
\draw(0,0) -- (12,0) -- (12,12) -- (0,12) -- (0,0);
\draw(0,4) -- (12,4);
\draw(0,8) -- (12,8);
\end{tikzpicture}
\\
\raisebox{0.4cm}{$4$}
&
\begin{tikzpicture}
[scale=0.08, thick]
\draw(0,0) -- (12,0) -- (12,12) -- 	(0,12) -- (0,0); 
\draw(6,0) -- (6,12);
\draw(3,0) -- (3,12);
\draw(1.5,0) -- (1.5,12);
\end{tikzpicture}
\quad
\begin{tikzpicture}
[scale=0.08, thick]
\draw(0,0) -- (12,0) -- (12,12) -- 	(0,12) -- (0,0); 
\draw(6,0) -- (6,12);
\draw(3,0) -- (3,12);
\draw(0,6) -- (3,6);
\end{tikzpicture}
\quad
\begin{tikzpicture}
[scale=0.08, thick]
\draw(0,0) -- (12,0) -- (12,12) -- 	(0,12) -- (0,0); 
\draw(6,0) -- (6,12);
\draw(3,0) -- (3,12);
\draw(4.5,0) -- (4.5,12);
\end{tikzpicture}
\quad
\begin{tikzpicture}
[scale=0.08, thick]
\draw(0,0) -- (12,0) -- (12,12) -- 	(0,12) -- (0,0); 
\draw(6,0) -- (6,12);
\draw(3,0) -- (3,12);
\draw(3,6) -- (6,6);
\end{tikzpicture}
\quad
\begin{tikzpicture}
[scale=0.08, thick]
\draw(0,0) -- (12,0) -- (12,12) -- 	(0,12) -- (0,0); 
\draw(6,0) -- (6,12);
\draw(3,0) -- (3,12);
\draw(6,6) -- (12,6);
\end{tikzpicture}

\quad
\begin{tikzpicture}
[scale=0.08, thick]
\draw(0,0) -- (12,0) -- (12,12) -- (0,12) -- (0,0);
\draw(6,0) -- (6,12);
\draw(9,0) -- (9,12);
\draw(10.5,0) -- (10.5,12);
\end{tikzpicture}
\quad
\begin{tikzpicture}
[scale=0.08, thick]
\draw(0,0) -- (12,0) -- (12,12) -- (0,12) -- (0,0);
\draw(6,0) -- (6,12);
\draw(9,0) -- (9,12);
\draw(9,6) -- (12,6);
\end{tikzpicture}
\quad
\begin{tikzpicture}
[scale=0.08, thick]
\draw(0,0) -- (12,0) -- (12,12) -- (0,12) -- (0,0);
\draw(6,0) -- (6,12);
\draw(9,0) -- (9,12);
\draw(7.5,0) -- (7.5,12);
\end{tikzpicture}
\quad
\begin{tikzpicture}
[scale=0.08, thick]
\draw(0,0) -- (12,0) -- (12,12) -- (0,12) -- (0,0);
\draw(6,0) -- (6,12);
\draw(9,0) -- (9,12);
\draw(6,6) -- (9,6);
\end{tikzpicture}
\quad
\begin{tikzpicture}
[scale=0.08, thick]
\draw(0,0) -- (12,0) -- (12,12) -- (0,12) -- (0,0);
\draw(6,0) -- (6,12);
\draw(9,0) -- (9,12);
\draw(0,6) -- (6,6);
\end{tikzpicture}
\\ &
\begin{tikzpicture}
[scale=0.08, thick]
\draw(0,0) -- (12,0) -- (12,12) -- (0,12) -- (0,0);
\draw(6,0) -- (6,12);
\draw(0,6) -- (6,6);
\draw(3,0) -- (3,6);
\end{tikzpicture}
\quad
\begin{tikzpicture}
[scale=0.08, thick]
\draw(0,0) -- (12,0) -- (12,12) -- (0,12) -- (0,0);
\draw(6,0) -- (6,12);
\draw(0,6) -- (6,6);
\draw(0,3) -- (6,3);
\end{tikzpicture}
\quad
\begin{tikzpicture}
[scale=0.08, thick]
\draw(0,0) -- (12,0) -- (12,12) -- (0,12) -- (0,0);
\draw(6,0) -- (6,12);
\draw(0,6) -- (6,6);
\draw(3,6) -- (3,12);
\end{tikzpicture}
\quad
\begin{tikzpicture}
[scale=0.08, thick]
\draw(0,0) -- (12,0) -- (12,12) -- (0,12) -- (0,0);
\draw(6,0) -- (6,12);
\draw(0,6) -- (6,6);
\draw(0,9) -- (6,9);
\end{tikzpicture}
\quad
\begin{tikzpicture}
[scale=0.08, thick]
\draw(0,0) -- (12,0) -- (12,12) -- (0,12) -- (0,0);
\draw(6,0) -- (6,12);
\draw(0,6) -- (6,6);
\draw(9,0) -- (9,12);
\end{tikzpicture}
\quad

\begin{tikzpicture}
[scale=0.08, thick]
\draw(0,0) -- (12,0) -- (12,12) -- (0,12) -- (0,0);
\draw(6,0) -- (6,12);
\draw(12,6) -- (6,6);
\draw(9,0) -- (9,6);
\end{tikzpicture}
\quad
\begin{tikzpicture}
[scale=0.08, thick]
\draw(0,0) -- (12,0) -- (12,12) -- (0,12) -- (0,0);
\draw(6,0) -- (6,12);
\draw(12,6) -- (6,6);
\draw(6,3) -- (12,3);
\end{tikzpicture}
\quad
\begin{tikzpicture}
[scale=0.08, thick]
\draw(0,0) -- (12,0) -- (12,12) -- (0,12) -- (0,0);
\draw(6,0) -- (6,12);
\draw(12,6) -- (6,6);
\draw(9,6) -- (9,12);
\end{tikzpicture}
\quad
\begin{tikzpicture}
[scale=0.08, thick]
\draw(0,0) -- (12,0) -- (12,12) -- (0,12) -- (0,0);
\draw(6,0) -- (6,12);
\draw(12,6) -- (6,6);
\draw(6,9) -- (12,9);
\end{tikzpicture}
\quad
\begin{tikzpicture}
[scale=0.08, thick]
\draw(0,0) -- (12,0) -- (12,12) -- (0,12) -- (0,0);
\draw(6,0) -- (6,12);
\draw(12,6) -- (6,6);
\draw(3,0) -- (3,12);
\end{tikzpicture}
\\ &
\begin{tikzpicture}
[scale=0.08, thick]
\draw(0,0) -- (12,0) -- (12,12) -- (0,12) -- (0,0);
\draw(0,6) -- (12,6);
\draw(0,3) -- (12,3);
\draw(0,1.5) -- (12,1.5);
\end{tikzpicture}
\quad
\begin{tikzpicture}
[scale=0.08, thick]
\draw(0,0) -- (12,0) -- (12,12) -- (0,12) -- (0,0);
\draw(0,6) -- (12,6);
\draw(0,3) -- (12,3);
\draw(6,0) -- (6,3);
\end{tikzpicture}
\quad
\begin{tikzpicture}
[scale=0.08, thick]
\draw(0,0) -- (12,0) -- (12,12) -- (0,12) -- (0,0);
\draw(0,6) -- (12,6);
\draw(0,3) -- (12,3);
\draw(0,4.5) -- (12,4.5);
\end{tikzpicture}
\quad
\begin{tikzpicture}
[scale=0.08, thick]
\draw(0,0) -- (12,0) -- (12,12) -- (0,12) -- (0,0);
\draw(0,6) -- (12,6);
\draw(0,3) -- (12,3);
\draw(6,3) -- (6,6);
\end{tikzpicture}
\quad
\begin{tikzpicture}
[scale=0.08, thick]
\draw(0,0) -- (12,0) -- (12,12) -- (0,12) -- (0,0);
\draw(0,6) -- (12,6);
\draw(0,3) -- (12,3);
\draw(6,6) -- (6,12);
\end{tikzpicture}
\quad

\begin{tikzpicture}
[scale=0.08, thick]
\draw(0,0) -- (12,0) -- (12,12) -- (0,12) -- (0,0);
\draw(0,6) -- (12,6);
\draw(0,9) -- (12,9);
\draw(0,10.5) -- (12,10.5);
\end{tikzpicture}
\quad
\begin{tikzpicture}
[scale=0.08, thick]
\draw(0,0) -- (12,0) -- (12,12) -- (0,12) -- (0,0);
\draw(0,6) -- (12,6);
\draw(0,9) -- (12,9);
\draw(6,9) -- (6,12);
\end{tikzpicture}
\quad
\begin{tikzpicture}
[scale=0.08, thick]
\draw(0,0) -- (12,0) -- (12,12) -- (0,12) -- (0,0);
\draw(0,6) -- (12,6);
\draw(0,9) -- (12,9);
\draw(0,7.5) -- (12,7.5);
\end{tikzpicture}
\quad
\begin{tikzpicture}
[scale=0.08, thick]
\draw(0,0) -- (12,0) -- (12,12) -- (0,12) -- (0,0);
\draw(0,6) -- (12,6);
\draw(0,9) -- (12,9);
\draw(6,6) -- (6,9);
\end{tikzpicture}
\quad
\begin{tikzpicture}
[scale=0.08, thick]
\draw(0,0) -- (12,0) -- (12,12) -- (0,12) -- (0,0);
\draw(0,6) -- (12,6);
\draw(0,9) -- (12,9);
\draw(6,0) -- (6,6);
\end{tikzpicture}
\\ &
\begin{tikzpicture}
[scale=0.08, thick]
\draw(0,0) -- (12,0) -- (12,12) -- (0,12) -- (0,0);
\draw(0,6) -- (12,6);
\draw(6,0) -- (6,6);
\draw(3,0) -- (3,6);
\end{tikzpicture}
\quad
\begin{tikzpicture}
[scale=0.08, thick]
\draw(0,0) -- (12,0) -- (12,12) -- (0,12) -- (0,0);
\draw(0,6) -- (12,6);
\draw(6,0) -- (6,6);
\draw(0,3) -- (6,3);
\end{tikzpicture}
\quad
\begin{tikzpicture}
[scale=0.08, thick]
\draw(0,0) -- (12,0) -- (12,12) -- (0,12) -- (0,0);
\draw(0,6) -- (12,6);
\draw(6,0) -- (6,6);
\draw(9,0) -- (9,6);
\end{tikzpicture}
\quad
\begin{tikzpicture}
[scale=0.08, thick]
\draw(0,0) -- (12,0) -- (12,12) -- (0,12) -- (0,0);
\draw(0,6) -- (12,6);
\draw(6,0) -- (6,6);
\draw(6,3) -- (12,3);
\end{tikzpicture}
\quad
\begin{tikzpicture}
[scale=0.08, thick]
\draw(0,0) -- (12,0) -- (12,12) -- (0,12) -- (0,0);
\draw(0,6) -- (12,6);
\draw(6,0) -- (6,6);
\draw(0,9) -- (12,9);
\end{tikzpicture}
\quad

\begin{tikzpicture}
[scale=0.08, thick]
\draw(0,0) -- (12,0) -- (12,12) -- (0,12) -- (0,0);
\draw(0,6) -- (12,6);
\draw(6,12) -- (6,6);
\draw(3,6) -- (3,12);
\end{tikzpicture}
\quad
\begin{tikzpicture}
[scale=0.08, thick]
\draw(0,0) -- (12,0) -- (12,12) -- (0,12) -- (0,0);
\draw(0,6) -- (12,6);
\draw(6,12) -- (6,6);
\draw(0,9) -- (6,9);
\end{tikzpicture}
\quad
\begin{tikzpicture}
[scale=0.08, thick]
\draw(0,0) -- (12,0) -- (12,12) -- (0,12) -- (0,0);
\draw(0,6) -- (12,6);
\draw(6,12) -- (6,6);
\draw(9,6) -- (9,12);
\end{tikzpicture}
\quad
\begin{tikzpicture}
[scale=0.08, thick]
\draw(0,0) -- (12,0) -- (12,12) -- (0,12) -- (0,0);
\draw(0,6) -- (12,6);
\draw(6,12) -- (6,6);
\draw(6,9) -- (12,9);
\end{tikzpicture}
\quad
\begin{tikzpicture}
[scale=0.08, thick]
\draw(0,0) -- (12,0) -- (12,12) -- (0,12) -- (0,0);
\draw(0,6) -- (12,6);
\draw(6,12) -- (6,6);
\draw(0,3) -- (12,3);
\end{tikzpicture}
\\ &
\begin{tikzpicture}
[scale=0.08, thick]
\draw(0,0) -- (12,0) -- (12,12) -- (0,12) -- (0,0);
\draw(4,0) -- (4,12);
\draw(8,0) -- (8,12);
\draw(2,0) -- (2,12);
\end{tikzpicture}
\quad
\begin{tikzpicture}
[scale=0.08, thick]
\draw(0,0) -- (12,0) -- (12,12) -- (0,12) -- (0,0);
\draw(4,0) -- (4,12);
\draw(8,0) -- (8,12);
\draw(0,6) -- (4,6);
\end{tikzpicture}
\quad
\begin{tikzpicture}
[scale=0.08, thick]
\draw(0,0) -- (12,0) -- (12,12) -- (0,12) -- (0,0);
\draw(4,0) -- (4,12);
\draw(8,0) -- (8,12);
\draw(6,0) -- (6,12);
\end{tikzpicture}
\quad
\begin{tikzpicture}
[scale=0.08, thick]
\draw(0,0) -- (12,0) -- (12,12) -- (0,12) -- (0,0);
\draw(4,0) -- (4,12);
\draw(8,0) -- (8,12);
\draw(4,6) -- (8,6);
\end{tikzpicture}
\quad
\begin{tikzpicture}
[scale=0.08, thick]
\draw(0,0) -- (12,0) -- (12,12) -- (0,12) -- (0,0);
\draw(4,0) -- (4,12);
\draw(8,0) -- (8,12);
\draw(10,0) -- (10,12);
\end{tikzpicture}
\quad
\begin{tikzpicture}
[scale=0.08, thick]
\draw(0,0) -- (12,0) -- (12,12) -- (0,12) -- (0,0);
\draw(4,0) -- (4,12);
\draw(8,0) -- (8,12);
\draw(8,6) -- (12,6);
\end{tikzpicture}
\quad

\begin{tikzpicture}
[scale=0.08, thick]
\draw(0,0) -- (12,0) -- (12,12) -- (0,12) -- (0,0);
\draw(0,4) -- (12,4);
\draw(0,8) -- (12,8);
\draw(6,0) -- (6,4);
\end{tikzpicture}
\quad
\begin{tikzpicture}
[scale=0.08, thick]
\draw(0,0) -- (12,0) -- (12,12) -- (0,12) -- (0,0);
\draw(0,4) -- (12,4);
\draw(0,8) -- (12,8);
\draw(0,2) -- (12,2);
\end{tikzpicture}
\quad
\begin{tikzpicture}
[scale=0.08, thick]
\draw(0,0) -- (12,0) -- (12,12) -- (0,12) -- (0,0);
\draw(0,4) -- (12,4);
\draw(0,8) -- (12,8);
\draw(6,4) -- (6,8);
\end{tikzpicture}
\quad
\begin{tikzpicture}
[scale=0.08, thick]
\draw(0,0) -- (12,0) -- (12,12) -- (0,12) -- (0,0);
\draw(0,4) -- (12,4);
\draw(0,8) -- (12,8);
\draw(0,6) -- (12,6);
\end{tikzpicture}
\\ &
\begin{tikzpicture}
[scale=0.08, thick]
\draw(0,0) -- (12,0) -- (12,12) -- (0,12) -- (0,0);
\draw(0,4) -- (12,4);
\draw(0,8) -- (12,8);
\draw(6,8) -- (6,12);
\end{tikzpicture}
\quad
\begin{tikzpicture}
[scale=0.08, thick]
\draw(0,0) -- (12,0) -- (12,12) -- (0,12) -- (0,0);
\draw(0,4) -- (12,4);
\draw(0,8) -- (12,8);
\draw(0,10) -- (12,10);
\end{tikzpicture}
\quad

\begin{tikzpicture}
[scale=0.08, thick]
\draw(0,0) -- (12,0) -- (12,12) -- (0,12) -- (0,0);
\draw(6,0) -- (6,12);
\draw(0,4) -- (6,4);
\draw(0,8) -- (6,8);
\end{tikzpicture}
\quad
\begin{tikzpicture}
[scale=0.08, thick]
\draw(0,0) -- (12,0) -- (12,12) -- (0,12) -- (0,0);
\draw(6,0) -- (6,12);
\draw(6,4) -- (12,4);
\draw(6,8) -- (12,8);
\end{tikzpicture}
\quad
\begin{tikzpicture}
[scale=0.08, thick]
\draw(0,0) -- (12,0) -- (12,12) -- (0,12) -- (0,0);
\draw(0,6) -- (12,6);
\draw(4,0) -- (4,6);
\draw(8,0) -- (8,6);
\end{tikzpicture}
\quad
\begin{tikzpicture}
[scale=0.08, thick]
\draw(0,0) -- (12,0) -- (12,12) -- (0,12) -- (0,0);
\draw(0,6) -- (12,6);
\draw(4,6) -- (4,12);
\draw(8,6) -- (8,12);
\end{tikzpicture}
\quad
\begin{tikzpicture}
[scale=0.08, thick]
\draw(0,0) -- (12,0) -- (12,12) -- (0,12) -- (0,0);
\draw(6,0) -- (6,12);
\draw(3,0) -- (3,12);
\draw(9,0) -- (9,12);
\end{tikzpicture}
\quad
\begin{tikzpicture}
[scale=0.08, thick]
\draw(0,0) -- (12,0) -- (12,12) -- (0,12) -- (0,0);
\draw(0,6) -- (12,6);
\draw(0,3) -- (12,3);
\draw(0,9) -- (12,9);
\end{tikzpicture}
\quad
\begin{tikzpicture}
[scale=0.08, thick]
\draw(0,0) -- (12,0) -- (12,12) -- (0,12) -- (0,0);
\draw(6,0) -- (6,12);
\draw(0,6) -- (12,6);
\end{tikzpicture}

\end{array}$

\caption{Elements in $\S_{2,n}$ for $n \leq 4$.}\label{Fig:0102}
\end{center}
\end{figure}

Let's introduce some notation so we can discuss these decompositions more precisely. Let $\mN = \set{1, 2, \ldots, }$ denote the set of natural numbers, and $[n] = \set{1, 2, \ldots, n}$ for every $n \in \mN$. Given a region
\[
R = ( a_1, b_1 ) \times \cdots \times ( a_d, b_d ) \subseteq (0,1)^d
\]
(all regions mentioned in this manuscript will be rectangular), a coordinate $i \in [d]$, and an arity $p \geq 2$, define
\[
c_j = a_i + \frac{j}{p} (b_i-a_i)
\]
for every $j \in \set{0,1,\ldots, p}$, and
\[
H_{i,p}(R) = \set{ \set{x \in R : c_{j-1} < x_i < c_{j}} : j \in [p]}
\]
Thus, $H_{i,p}(R)$ is a set consisting of the $p$ regions obtained from splitting $R$ along the $i^{\tn{th}}$ coordinate. 
%We call $H_{i,p}(R)$ the \emph{$p$-splitting} of $R$ in the coordinate $i$.
Then we define the set of hypercube decompositions $\S_{d}$ recursively as follows:

\begin{itemize}
\item
$\set{(0,1)^d} \in \S_d$ --- this is the trivial decomposition with one region, the entire unit hypercube.
\item
For every $S \in \S_d$, region $R \in S$, coordinate $i \in [d]$, and arity $p \geq 2$,
\[
(S \setminus \set{R}) \cup H_{i,p}(R) \in \S_d.
\]
\end{itemize}

In other words, $\S_d$ consists of the decompositions of $(0,1)^d$ that can be achieved by a sequence of splitting operations. In particular, the coordinate and arity used in each splitting operation are arbitrary and can vary over the splitting sequence. Furthermore, given $S \in \S_d$ we let $|S|$ denote the number of regions in $S$, and define 
\[
\S_{d,n} = \set{ S \in \S_d : |S| = n}
\]
for every integer $n \geq 1$. Note that distinct splitting sequences can result in the same decomposition. For an example, the decomposition $S = H_{1,6}( (0,1) ) \in \S_{1,6}$ can be obtained from simply $6$-splitting the unit interval $(0,1)$, or $2$-splitting $(0,1)$ followed by $3$-splitting each of $(0,1/2)$ and $(1/2,1)$. 

We are interested in enumerating the number of decompositions $s_d(n) = |\S_{d,n}|$. For small values of $d$, we obtain the following sequences:
\[
\begin{array}{l|rrrrrrrrrrr}
n & 1 & 2 & 3 & 4 & 5 & 6 & 7 & 8 & 9 & 10 & \cdots \\
\hline
\href{https://oeis.org/A050385}{s_1(n)} & 1 & 1 & 3 & 10 & 39 & 160 & 691 & 3081 & 14095 & 65757 & \cdots \\
% 311695               1496833               7266979              35608419             175875537
s_2(n) & 1 & 2 & 10 & 59 & 394 & 2810 & 20998 & 162216  & 1285185 & 10384986 & \cdots \\
%  85258446             709138185            5962927656           50606595570          432920891516
s_3(n) & 1 & 3 & 21 & 177 & 1677 & 17001 & 180525 & 1981909 & 22314339 & 256245783 & \cdots 
%         2989663689        35338763001       422295624651      5093326567395     61921421632673
\end{array}
\]
Note that the hyperlink at $s_1(n)$ in the table above directs to the sequence's entry in the On-Line Encyclopedia of Integer Sequences (OEIS)~\cite{OEIS}. Similar hyperlinks are placed throughout this manuscript for all integer sequences that are already in the OEIS at the time of this writing.

\subsection{A generalized M\"obius function}

To state our main result, we will need to describe the following generalization of the well-studied M\"obius function. Given $n \in \mN$, the \emph{M\"{o}bius function} is defined to be
\[
\mu(n) = 
\begin{cases}
(-1)^k & \tn{if $n$ is a product of $k$ distinct primes;}\\
0 & \tn{if $p^2 | n$ for some $p>1$.}
\end{cases}
\]
(The convention is that $\mu(1) = 1$.) A well-known property of $\mu(n)$ is that
\begin{equation}\label{Eq:0101}
\sum_{i | n} \mu(i) =
\begin{cases}
1 & \tn{if $n=1$;}\\
0 & \tn{if $n > 1$.}
\end{cases}
\end{equation}
More generally, given two arithmetic functions $\a, \beta : \mN \to \mR$, their \emph{Dirichlet convolution} $\a \ast \beta : \mN \to \mR$ is the function
\[
(\a \ast \beta)(n) = \sum_{i | n} \a(i) \beta \left(\frac{n}{i}\right).
\]
For an example, let $\mathbbm{1} : \mN \to \mR$ be the function where $\mathbbm{1}(n) = 1$ for all $n \geq 1$, and let $\delta : \mN \to \mR$ denote the function where $\delta(1)=1$ and $\delta(n) = 0$ for all $n \geq 2$. Then \eqref{Eq:0101} can be equivalently stated as
\[
\mu \ast \mathbbm{1} = \delta.
\]
Next, given integer $d \geq 1$, we define the \emph{$d$-fold M\"obius function} where
\[
\mu_d = \underbrace{ \mu \ast \mu \ast \cdots \ast \mu}_{\tn{$d$ times}}.
\]
The following table gives the first few terms of $\mu_d(n)$ for $d \leq 3$.
\[
\begin{array}{l|rrrrrrrrrrrrrrrr}
n & 1 & 2 & 3 & 4 & 5 & 6 & 7 & 8 & 9& 10 &11 & 12 & 13 & 14 & 15& \cdots \\
\hline
\href{https://oeis.org/A008683}{\mu_1(n)} & 1 & -1 & -1 & 0 & -1 & 1 & -1 & 0& 0 & 1 & -1 &  0&  -1&   1&   1& \cdots \\
%   0  -1   0  -1 0
\href{https://oeis.org/A007427}{\mu_2(n)} & 1 & -2 & -2 & 1 & -2 & 4 & -2 & 0& 1 & 4 &  -2  &-2 & -2&   4 &  4 & \cdots \\
% 0  -2  -2  -2 -2
\mu_3(n) & 1 & -3 & -3 & 3 & -3 & 9 & -3 & -1& 3 & 9 & -3  & -9 &  -3 &   9 &   9& \cdots 
% 0   -3   -9   -3   -9
\end{array}
\]
The generalized M\"obius function $\mu_d$ was independently discovered several times, first by Fleck in 1915~\cite[Section 2.2]{SandorC04}. The reader may refer to~\cite{SandorC04} for the historical developments of this (and other) generalizations of $\mu$.

Next, we state two formulas for $\mu_d(n)$ that will be useful subsequently.

\begin{lemma}
For all integers $d,n \geq 1$
\begin{equation}\label{Eq:0102}
\mu_d(n) = \sum_{\substack{ n_1, \ldots, n_d \geq 1 \\ \prod_{i=1}^d n_i = n}} \mu(n_i)
\end{equation}
Moreover, suppose $n$ has prime factorization $n=p_1^{m_1} p_2^{m_2} \cdots p_k^{m_k}$. Then
\begin{equation}\label{Eq:0103}
\mu_d(n) = \prod_{i=1}^k (-1)^{m_i} \binom{d}{m_i}.
\end{equation}
\end{lemma}

\begin{proof}
\eqref{Eq:0102} follows immediately from the definition of the Dirichlet convolution and a simple induction on $d$. For~\eqref{Eq:0103}, one can use induction on $d$ to show that $\mu_d(p_1^{m_1}) = (-1)^{m_1} \binom{d}{m_i}$ for all prime powers $p_1^{m_1}$, and then use the facts that
\begin{itemize}
\item
$\mu$ is multiplicative (i.e., $\mu(ab) = \mu(a)\mu(b)$ if $a,b$ are coprime);
\item
if $\a,\beta$ are multiplicative, so is their Dirichlet convolution $\a \ast \beta$.
\end{itemize}
Thus, $\mu_d$ is multiplicative as well, and so it follows that
\[
\mu_d(n) = \mu_d(p_1^{m_1} \cdots p_k^{m_k}) = \prod_{i=1}^k (-1)^{m_i} \binom{d}{m_i}.
\]
\end{proof}

Our main result of the manuscript is the following.

\begin{theorem}\label{Thm:0101}
Given $d \in \mN$, define $y = \sum_{n \geq 1} s_d(n)x^n$. Then $y$ satisfies the functional equation
\begin{equation}
x = \sum_{n \geq 1} \mu_d(n)y^n.
\end{equation}
\end{theorem}

\subsection{Natural exact covering systems}\label{Sec:0103}

Theorem~\ref{Thm:0101} generalizes a recent result on natural exact covering systems, which we describe here. Given $a \in \mZ, n \in \mN$, we let $\res{a}{n}$ denote the residue class $\set{ x \in \mZ: x \equiv a~\tn{mod}~n}$. Given a residue class $\res{a}{n}, r \in \mN$, and $i \in \set{0,1,\ldots, r-1}$, define
\[
E_{i,r} (\res{a}{n}) = \res{in + a}{rn}.
\]
Observe that 
\begin{equation}\label{Eq:010a}
\set{ E_{i,r}(\res{a}{n}) : i \in \set{0, \ldots, r-1}}
\end{equation}
 partition $\res{a}{n}$, and so we can think of~\eqref{Eq:010a} as an $r$-splitting of $\res{a}{n}$. We then define the set of natural exact covering systems (NECS) $\C$ recursively as follows:

\begin{itemize}
\item
$\set{\res{0}{1}} \in \C$ --- this is the trivial NECS with one residual class.
\item
For all $C \in \C$, residue class $\res{a}{n} \in C$, and integer $r \geq 2$,
\[
(C \setminus \set{ \res{a}{n} }) \cup \set{ E_{i,r}(\res{a}{n}) : i \in \set{0,\ldots, r-1}}\in \C.
\]
\end{itemize}

Notice that $\res{0}{1} = \mZ$, and each NECS $C \in \C$ consists of a collection of residue classes that partition $\mZ$. Also, given $n \in \mN$, let $\C_{n} \subseteq \C$ denote the set of NECS with exactly $n$ residue classes, and let $c(n) = |\C_n|$. For example, 
\[
\set{\res{0}{4}, \res{2}{8}, \res{6}{8}, \res{1}{6}, \res{3}{6}, \res{5}{6}}
\]
is an element in $\C_6$ as it can be obtained from $2$-splitting $\res{0}{1}$, then $2$-splitting $\res{0}{2}$, then $2$-splitting $\res{2}{4}$, then $3$-splitting $\res{1}{2}$. Recently, Goulden, Granville, Richmond, and Shallit~\cite{GouldenGRS18} showed the following:

\begin{theorem}\label{Thm:0102}
Let $y = \sum_{n \geq 1} c(n) x^n$. Then $y$ satisfies the function equation
\[
x = \sum_{n \geq 1} \mu(n)y^n.
\]
\end{theorem}

Not only does Theorem~\ref{Thm:0102} provide a formula for the number of NECS with a given residual classes, it also shows that the generating series with coefficients $c(n)$ is exactly the compositional inverse of the familiar M\"obius power series. In that light, Theorem~\ref{Thm:0101} is a generalization of Theorem~\ref{Thm:0102}, as we provide a formula that gives $s_d(n)$ while also showing that they are the coefficients of the compositional inverse of the generalized M\"obius series. Moreover, Theorem~\ref{Thm:0101} implies that $s_1(n) = c(n)$ for every $n \geq 1$ (i.e., the number of ways to split $(0,1)$ into $n$ subintervals using a sequence of splitting operations is equal to the number of NECS of $\mZ$ with $n$ classes). 

To the best of our knowledge, NECS first appeared in the mathematical literature in Porubsk{\`y}'s work in 1974~\cite{Porubsky74}. More broadly, a set of residual classes $C = \set{ \res{a_i}{n_i} : i \in [k]}$ is a \emph{covering system} of $\mZ$ if $\bigcup_{i=1}^k \res{a_i}{n_i} = \mZ$ (i.e., the residual classes need not be disjoint or result from a sequence of splitting operations). An ample amount of literature has been dedicated to studying covering systems since they were introduced by Erd\H{o}s~\cite{Erdos50}. The reader may refer to~\cite{PorubskyS02, Sun21, ZaleskiZ20, Znam82} and others for broader expositions on the topic.

\subsection{Roadmap of the manuscript}

This manuscript is organized as follows.
In Section~\ref{Sec:02}, we introduce some helpful notions (such as the gcd of a decomposition) and prove Theorem~\ref{Thm:0101}.
Then in Section~\ref{Sec:03}, we take a small detour to study $a_d(n)$, the coefficients of the multiplicative inverse of the generalized M\"obius series. We establish a combinatorial interpretation for $a_d(n)$, and prove some results that we'll rely on in Section~\ref{Sec:04}, where we prove an asymptotic formula for $s_d(n)$ and study its growth rate. Finally, we establish a bijection between $\S_{1,n}$ and $\C_n$ in Section~\ref{Sec:05}, and give an example of how studying hypercube decompositions can lead to results on NECS.

\section{Proof of Theorem~\ref{Thm:0101}}\label{Sec:02}

In this section, we prove Theorem~\ref{Thm:0101} using a $d$-fold variant of M\"obius inversion. We remark that the structure of our proof is somewhat similar to Goulden et al.'s corresponding argument for NECS~\cite[Theorem 3]{GouldenGRS18}.

Given integers $r_1, \ldots, r_d \in \mN$, let $D_{(r_1, \ldots, r_d)} \in \S_d$ denote the decomposition obtained by 
\begin{itemize}
\item
starting with $(0,1)^d$;
\item
for $i= 1,\ldots, d$, $r_i$-split all regions in the current decomposition in coordinate $i$.
\end{itemize}

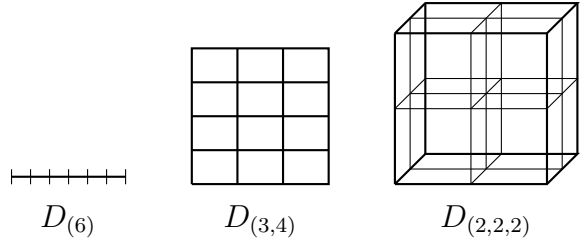
\begin{figure}[h!]
\begin{center}
\begin{tabular}{ccccc}
\begin{tikzpicture}
[scale=0.25]
\def\x{0.4}
\draw[thick] (0,0) -- (6,0);
\draw (0,{\x}) -- (0, {-\x});
\draw (1,{\x}) -- (1, {-\x});
\draw (2,{\x}) -- (2, {-\x});
\draw (3,{\x}) -- (3, {-\x});
\draw (4,{\x}) -- (4, {-\x});
\draw (5,{\x}) -- (5, {-\x});
\draw (6,{\x}) -- (6, {-\x});
\end{tikzpicture}
&\quad &
\begin{tikzpicture}
[scale=0.15, thick]
\draw[thick](0,0) -- (12,0) -- (12,12) -- (0,12) -- (0,0);
\draw(0,3) -- (12,3);
\draw(0,6) -- (12,6);
\draw(0,9) -- (12,9);
\draw(4,0) -- (4,12);
\draw(8,0) -- (8,12);
\end{tikzpicture}
&\quad &
\begin{tikzpicture}
[scale=0.2]
%increase to vertically lengthen cube, decrease to flatten
\def\ys{1} 
%axes
%\draw[thick, ->] (0,0) -- (12.5,0) node[anchor = west] {$1$};
%\draw[thick, ->] (0,0) -- (3, {3*\ys}) node[anchor = south] {$2$};
%\draw[thick, ->] (0,0) -- (0,12.5) node[anchor = south] {$3$};
%cube
\draw[thick] (0,0) -- (10,0) -- ( 12, {2* \ys}) -- (2,{2* \ys}) -- (0,0);
\draw[thick] (0,10) -- (10,10) -- (12,{10 + 2* \ys}) -- (2,{10 + 2* \ys}) -- (0,10);
\draw[thick] (0,0) -- (0,10);
\draw[thick] (10,0) -- (10,10);
\draw[thick] (12,{2* \ys}) -- (12,{10 + 2* \ys});
\draw[thick] (2,{2* \ys}) -- (2,{10 + 2* \ys});
%partition lines
\draw (5,0) -- (7,{2* \ys}) -- (7,{10 + 2* \ys}) -- (5,10) -- (5,0);
\draw (1,{\ys}) -- (11,{\ys}) -- (11,{10+\ys}) -- (1,{10+\ys}) -- (1,{\ys});
\draw (6,{\ys}) -- (6,{10+\ys});
\draw (5,5) -- (7,{5+2*\ys});
\draw (0,5) -- (10,5) -- (12, {5+2*\ys}) -- (2,{5+2*\ys}) -- (0,5);
\draw (1,{5+\ys}) -- (11,{5+\ys});
\end{tikzpicture}\\
$D_{(6)}$ 
&\quad&
$D_{(3,4)}$
&\quad&
$D_{(2,2,2)}$
\end{tabular}
\caption{Examples of decompositions of the form $D_{(r_1,\ldots, r_n)}$.}\label{Fig:0201}
\end{center}
\end{figure}

Figure~\ref{Fig:0201} illustrates a few examples of these decompositions. Notice that $D_{(r_1, \ldots, r_d)}$ has $\prod_{i=1}^d r_i$ regions of identical shape and size. Next, given $S,S' \in \S_d$, we say that $S'$ \emph{refines} $S$ --- denoted $S' \succeq S$ --- if $S'$ can be obtained from $S$ by a (possibly empty) sequence of splitting operations. For example, in Figure~\ref{Fig:0202}, we have $S_3 \succeq S_1$ and $S_3 \succeq S_2$, while $S_1, S_2$ are incomparable. It is not hard to see that $\succeq$ imposes a partial order on $\S_d$.

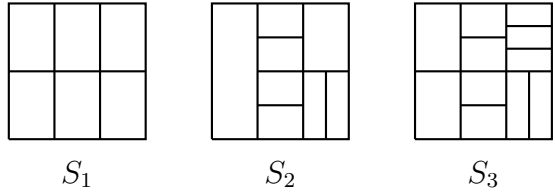
\begin{figure}[ht]
\begin{center}
\begin{tabular}{ccccc}
\begin{tikzpicture}
[scale=0.15, thick]
\draw(0,0) -- (12,0) -- (12,12) -- (0,12) -- (0,0);
\draw(4,0) -- (4,12);
\draw(8,0) -- (8,12);
\draw(0,6) -- (12,6);
\end{tikzpicture}
&\quad&
\begin{tikzpicture}
[scale=0.15, thick]
\draw(0,0) -- (12,0) -- (12,12) -- (0,12) -- (0,0);
\draw(4,0) -- (4,12);
\draw(8,0) -- (8,12);
\draw(4,6) -- (8,6);
\draw(4,3) -- (8,3);
\draw(4,9) -- (8,9);
\draw(12,6) -- (8,6);
\draw(10,0) -- (10,6);
\end{tikzpicture}
&\quad&
\begin{tikzpicture}
[scale=0.15, thick]
\draw(0,0) -- (12,0) -- (12,12) -- (0,12) -- (0,0);
\draw(4,0) -- (4,12);
\draw(8,0) -- (8,12);
\draw(0,6) -- (12,6);
\draw(4,3) -- (8,3);
\draw(4,9) -- (8,9);
\draw(10,0) -- (10,6);
\draw(8,8) -- (12,8);
\draw(8,10) -- (12,10);
\end{tikzpicture}\\
$S_1$ && $S_2$ && $S_3$
\end{tabular}
\caption{Illustrating the refinement relation $\succeq$ on $\S_d$.}\label{Fig:0202}
\end{center}
\end{figure}

Also, it will be convenient to define the notion of scaling. Given regions $R, R' \subseteq (0,1)^d$, notice there is a unique affine function $L_{R \to R'}: \mR^d \to \mR^d$ that preserves the lexicographical order of vectors, while satisfying $\set{ L_{R \to R'}(x) : x \in R} = R'$. More explicitly, let 
\begin{align*}
R &= (a_1, b_1) \times \cdots \times (a_d, b_d), \\
R' &= (a'_1, b'_1) \times \cdots \times (a'_d, b'_d).
\end{align*}
Then $L_{R \to R'} : \mR^d \to \mR^d$ where
\[
\left[L_{R \to R'}(x) \right]_i = a'_i + \frac{b'_i-a'_i}{b_i-a_i}(x-a_i)
\]
for every $i \in [d]$ is the function that satisfies the aforementioned properties. Next, given regions $R, R', R'' \subseteq (0,1)^d$, we define
\[
\scale_{R \to R'}(R'') = \set{L_{R \to R'}(x) : x \in R''}.
\]
We will always apply this scaling function in cases where $R'' \subseteq R$. Thus, intuitively, $\scale_{R \to R'}(R'')$ returns a region that's contained in $R'$ such that the relative position of $\scale_{R \to R'}(R'')$ inside $R'$  is the same of that of $R''$ inside $R$. More importantly, given a decomposition $S \in \S_d$, 
\[
\set{ \scale_{(0,1)^d \to R'}(R'') : R'' \in S}
\]
is a collection of sets that can be obtained from iteratively splitting $R'$. Conversely, if $\set{R_1,\ldots, R_n}$ is a collection of sets obtained from applying a sequence of splitting operations to region $R$, then 
\[
\set{ \scale_{R \to (0,1)^d}(R_i) : i \in [n]}
\]
is an element of $\S_{d,n}$. We then have the following:

\begin{lemma}\label{Lem:0201}
Let $y = \sum_{n \geq 1} s_{d}(n) x^n$. Then for all integers $r_1,\ldots, r_d \in \mN$,
\[
\sum_{\substack{S \in \S_d \\ S \succeq D_{(r_1, \ldots, r_d)}}} x^{|S|} = y^{ \prod_{i=1}^d r_i}.
\]
\end{lemma}

\begin{proof}
For convenience, let $n = \prod_{i=1}^d r_i$. First, observe the following bijection between $\set{S \in \S_d : S \succeq D_{(r_1, \ldots, r_d)}}$ and $\S_d^{n}$: Let $B_1, \ldots, B_n$ be the regions in $D_{(r_1, \ldots, r_d)}$. Then, given $n$ decompositions $S_1, \ldots, S_n \in \S_d$, 
\[
S = \bigcup_{j=1}^n \set{ \scale_{(0,1)^d \to B_j}(R) : R \in S_j}
\]
gives a decomposition of $(0,1)^d$ that refines $D_{(r_1, \ldots, r_d)}$. On the other hand, given $D \in \S_d$ that refines $D_{r_1,\ldots, r_n}$, define $D_j \subseteq S$ such that
\[
D_j = \set{ R \in S : R \subseteq B_j}
\]
for every $j \in [n]$. By assumption that $S \succeq D_{(r_1,\ldots,r_d)}$, we know that $D_j$ can be obtained from applying a sequence of splitting operations to $B_j$ for every $j$. Hence, each of 
\[
S_j = \set{ \scale_{B_j \to (0,1)^n} (R) : R \in D_j}
\]
is an element of $\S_d$. Doing so for each of $D_1, \ldots, D_n$ results in an $n$-tuple of decompositions that corresponds to $S$. Thus, it follows that
\[
 y^n = \sum_{S_1, \ldots, S_n \in \S_d} x^{\sum_{i=1}^n |S_i|} =  \sum_{\substack{ S \in \S_d \\ S \succeq D_{(r_1, \ldots, r_d)}}} x^{|S|}.
\]
\end{proof}

Next, given $S \in \S_d$, we say that $\gcd(S) = (r_1, \ldots, r_d)$ if
\begin{itemize}
\item
$S \succeq D_{(r_1, \ldots, r_d)}$;
\item
there doesn't exist $(r_1',\ldots, r_d') \neq (r_1, \ldots, r_d)$ where $(r_1',\ldots, r_d') \geq (r_1, \ldots, r_d)$ and $S \succeq D_{(r_1', \ldots, r_d')}$.
\end{itemize}

For example, in Figure~\ref{Fig:0202}, $\gcd(S_1) = \gcd(S_3) = (3,2)$, while $\gcd(S_2) = (3,1)$. As we shall see later, this notion of the gcd of a decomposition corresponds well to the existing notion of the gcd of an NECS. Next, define 
\[
s_{d, (r_1,\ldots, r_d)}(n) = \left| \set{S \in \S_{d,n} : \gcd(S) = (r_1, \ldots, r_d) } \right|,
\]
and the series
\[
S_{d,(r_1,\ldots, r_d)}(x) = \sum_{n \geq 1} s_{d, (r_1,\ldots, r_d)}(n) x^n.
\]
In particular, consider the case when $r_1 = \cdots = r_d = 1$. If $S \in \S_{d,n}$ for some $n \geq 2$, then the splitting sequence for $S$ is non-empty, and so the $\gcd(S) \neq (1, \ldots, 1)$. Thus, we see that the only decomposition $S$ where $\gcd(S) = (1, \ldots, 1)$ is the trivial decomposition $\set{ (0,1)^d}$. This implies that 
\[
s_{d, (1, \ldots, 1)}(n) = 
\begin{cases}
1 & \tn{if $n=1$;}\\
0 & \tn{otherwise.}
\end{cases}
\]
Hence, it follows that $S_{d,(1,\ldots, 1)}(x) = x$. For general $r_1, \ldots, r_d$, we have the following.

\begin{lemma}\label{Lem:0202}
For all $d,n, r_1, \ldots, r_d \in \mN$,
\[
\sum_{\substack{S \in \S_d \\  S \succeq D_{(r_1, \ldots, r_d)}}} x^{|S|} = \sum_{a_1, \ldots, a_d \geq 1} S_{d,(a_1r_1,\ldots, a_dr_d)}(x) 
\]
\end{lemma}

\begin{proof}
It suffices to show that, given fixed $d,n$, and $r_1, \ldots, r_d$,
\[
\set{ S \in \S_{d,n} : S \succeq D_{(r_1, \ldots, r_d)}} = \bigcup_{a_1, \ldots, a_d \geq 1} \set{ S \in \S_{d,n} : \gcd(S) = (a_1r_1, \ldots, a_dr_d)},
\]
since the sets making up the union on the right hand side are mutually disjoint.

For $\supseteq$, notice that if $\gcd(S) = (a_1r_1, \ldots, a_dr_d)$, then $S \succeq D_{(r_1, \ldots, r_d)}$. It is also obvious that $D_{(a_1r_1, \ldots, a_dr_d)} \succeq D_{(r_1, \ldots, r_d)}$ for all $a_1, \ldots, a_d \geq 1$. Since $\succeq$ is a transitive relation, the containment holds.

We next prove $\subseteq$. Let $S \succeq D_{(r_1, \ldots, r_d)}$ and suppose $\gcd(S) = (b_1,\ldots, b_d)$. Therefore, for every coordinate $i \in [d]$, $S \succeq \set{ H_{i, r_i}((0,1)^d)}$ and $S \succeq \set{ H_{i, b_i}((0,1)^d)}$. Thus, if we let $\ell_i$ be the least common multiple of $r_i$ and $b_i$, it follows that $S \succeq \set{ H_{i, \ell_i}((0,1)^d)}$. Since $b_i$ was chosen maximally (by the definition of gcd), it must be that $b_i = \ell_i$ for every $i$. This means that $r_i | b_i$, and so there exists $a_i \in \mN$ where $b_i = a_ir_i$. This finishes our proof.
\end{proof}

With Lemmas~\ref{Lem:0201} and~\ref{Lem:0202}, we are now ready to prove a result that is slightly more general than Theorem~\ref{Thm:0101}.

\begin{theorem}\label{Thm:0201}
Let $d, r_1,\ldots, r_d \in \mN$, and $y = \sum_{n \geq 1} s_d(n)x^n$. Then
\[
S_{d, (r_1,\ldots, r_d)}(x)  = \sum_{n \geq 1} \mu_d(n) y^{\left(\prod_{i=1}^d r_i\right)n}.
\]
\end{theorem}

\begin{proof}
From Lemmas~\ref{Lem:0201} and~\ref{Lem:0202}, we obtain that for every fixed $n \geq 1$,
\begin{equation}\label{Eq:0201}
y^{\left(\prod_{i=1}^d r_i \right)n} = \sum_{a_1, \ldots, a_d \geq 1} S_{d,(a_1r_1,\ldots, a_dr_dn)}(x). 
\end{equation}
In fact, we see that $y^{\left(\prod_{i=1}^d r_i \right)n} = \sum_{a_1, \ldots, a_d \geq 1} S_{d,(a_1q_1,\ldots, a_dq_d)}(x)$ for all $q_1, \ldots, q_d \in \mN$ where
\begin{equation}\label{Eq:0202}
\prod_{i=1}^d q_i = \left(\prod_{i=1}^d r_i \right)n.
\end{equation}
Now we multiply both sides of~\eqref{Eq:0201} by $\mu_d(n)$ and sum them over $n \geq 1$. On the left hand side, we obtain 
\[
\sum_{n \geq 1} \mu_d(n) y^{\left(\prod_{i=1}^d r_i\right)n}.
\]
For the right hand side, we have
\begin{align*}
& \sum_{n \geq 1} \mu_d(n)  \left( \sum_{a_1, \ldots, a_d \geq 1} S_{d,(a_1r_1,\ldots, a_dr_dn)}(x) \right)\\
\overset{\eqref{Eq:0102}}{=}{}& \sum_{n \geq 1} \sum_{\substack{n_1, \ldots, n_d \geq 1 \\ n_1n_2\cdots n_d=n}} \prod_{i=1}^d \mu(n_i)\left( \sum_{a_1, \ldots, a_d \geq 1} S_{d,(a_1r_1,\ldots, a_dr_dn)}(x) \right)\\
={}&  \sum_{n_1, \ldots, n_d \geq 1} \prod_{i=1}^d \mu(n_i)\left( \sum_{a_1, \ldots, a_d \geq 1} S_{d,(a_1r_1,\ldots, a_dr_d \left( \prod_{i=1}^d n_i \right))}(x) \right)\\
\overset{\eqref{Eq:0202}}{=}{}& \sum_{n_1, \ldots, n_d \geq 1} \prod_{i=1}^d \mu(n_i)\left( \sum_{a_1, \ldots, a_d \geq 1} S_{d,(n_1a_1r_1,\ldots, n_da_dr_d)}(x) \right)\\
%&=& \sum_{n_1, \ldots, n_d \geq 1} \prod_{i=1}^d \mu(n_i)\left( \sum_{a_1, \ldots, a_d \geq 1} S_{d,(n_1a_1r_1,\ldots, n_da_dr_d)}(x) \right)\\
={}& \sum_{b_1, \ldots, b_d \geq 1} S_{d,(b_1r_1,\ldots, b_dr_d)}(x) \left( \sum_{n_1 | b_1} \mu(n_1) \right) \cdots \left( \sum_{n_d | b_d} \mu(n_d) \right) \\
\overset{\eqref{Eq:0101}}{=}{}& S_{d, (r_1,\ldots, r_d)}(x).
\end{align*}
This finishes the proof.
\end{proof}

When $r_1= \cdots = r_d = 1$, Theorem~\ref{Thm:0201} specializes to $x = \sum_{n \geq 1} \mu_d(n) y^n$, and thus Theorem~\ref{Thm:0101} follows as a consequence.

\section{A Small Detour: The Sequences $a_d(n)$}\label{Sec:03}

Observe that, if we define the series
\[
M_d(z) = \sum_{n \geq 1} \mu_d(n)z^n,
\]
then the functional equation in Theorem~\ref{Thm:0101} can be restated simply as $M_d(y) = x$. We have shown earlier that $s_d(n)$ gives the coefficient of $x^n$ in $y$, hence finding a combinatorial interpretation of the coefficients of the compositional inverse of the series $M_d(z)$. In this section, we'll mainly focus on the coefficients of what's essentially the multiplicative inverse of $M_d(z)$. For every $d \geq 1$ and $n \geq 0$, define
\begin{equation}\label{Eq:0301}
a_d(n) = [z^n] \frac{z}{M_d(z)}.
\end{equation}
Then it follows that
\[
M_d(z) \left( \sum_{i \geq 0} a_d(i)z^i \right) = z ~\implies~ \sum_{k = 1}^n \mu_d(k) a_d(n-k) = [z^n] z.
\]
Since $\mu_d(1) = 1$ for all $d$, $a_d(n)$ satisfies the recurrence relation 
\begin{equation}\label{Eq:0302}
a_d(n) = \sum_{k =2}^{n+1} -\mu_d(k)a_d(n+1-k)
\end{equation}
for all $n \geq 1$, with the initial condition $a_d(0) = 1$. The following table gives some values of $a_d(n)$ for small $d$ and $n$:
\[
\begin{array}{l|rrrrrrrrrrrr}
n & 0 & 1 & 2 & 3 & 4 & 5 & 6 & 7 & 8 & 9 & 10 &\cdots \\
\hline
\href{https://oeis.org/A073776}{a_1(n)} &
1 & 1 & 2 & 3 & 6 & 9 & 17 & 28 & 50 & 83 & 147 & \cdots \\
%  & 249     435        742       1288       2207       3819       6561
a_2(n) &
1 & 2 & 6 & 15 & 42 & 108 & 291 & 766 & 2041 & 5395&  14328 & \cdots \\
%              37956    100689           266886           707743
a_3(n) &
1 & 3 & 12 & 42 & 156 & 558 & 2028 & 7318 & 26490 & 95730 & 346218 & \cdots 
%                  1251690               4526004             16364346             59169592
 \end{array}
 \]
The main goal of this section is to establish two results (Proposition~\ref{Prop:0301} and Lemma~\ref{Lem:0303}) that we will need when we study the asymptotic behavior and growth rate of $s_d(n)$. While we do so, we will take somewhat of a scenic route and uncover some properties of $a_d(n)$ along the way. 
 
\subsection{A lower bound for $a_d(n+1)/a_d(n)$}

When we prove an asymptotic formula for $s_d(n)$ in Section~\ref{Sec:04}, one of the requirements will be to show that $a_d(n) \geq 0$ for every $d \geq 1$ and $n \geq 0$. Likewise, Goulden et al.~\cite[Theorem 2]{GouldenGRS18} showed that $a_1(n) \geq 0$ with a complex analysis argument as a part of their work establishing an asymptotic formula for $c(n)$.

In this section, we describe the set $\tilde{\A}_{d,n}$ and show that $|\tilde{\A}_{d,n}| = a_d(n)$ for all $d \geq 1$ and $n\geq 0$. Having a combinatorial interpretation for $a_d(n)$ implies (among other things) that these coefficients are indeed nonnegative. Before we are able to define $\tilde{\A}_{d,n}$, we first need to introduce some intermediate objects $\B_{d,n}, \A_{d,n}$ and mention some of their relevant properties along the way. Given $d \in \mN$, we call a multiset $S= \set{s_1, \ldots, s_{\l}}$ a \emph{$d$-coloured prime set} if the following holds:
\begin{itemize}
\item
The elements $s_1, \ldots s_{\l}$ are (not necessarily distinct) prime numbers.
\item
Each of $s_1, \ldots, s_{\l}$ is assigned a colour from $[d]$, such that each instance of the same prime number must be assigned distinct colours.
\end{itemize}
We also define the weight of $S$ to be $w(S) = \prod_{i=1}^{\l} s_i  -1$, and the sign of $S$ to be $v(S) = (-1)^{\l+1}$. For every $n \geq 0$, let $\B_{d,i}$ denote the set of $d$-coloured prime sets with weight $i$. For instance, 
\[
\B_{2,11} = \set{ \set{2_1, 2_2, 3_1}, \set{2_1, 2_2, 3_2} }
\]
(where the assigned colour of a number is indicated in its subscript), and $\B_{2, 26} = \emptyset$. More generally, notice that if $n+1$ has prime factorization $n+1= p_1^{m_1}\cdots p_{k}^{m_k}$, then every element in $\B_{d,n}$ must contain exactly $m_i$ copies of $p_i$ for every $i \in [k]$. This implies that
\begin{itemize}
\item
$|\B_{d,n}| = \prod_{i=1}^{k} \binom{d}{m_i} = | \mu_d(n+1)|$, since there are $\binom{d}{m_i}$ ways to colour to the $m_i$ copies of $p_i$ for every $i \in [k]$;
\item
$v(B) = (-1)^{1+\sum_{i=1}^k m_i}$ for every element in $\B_{d,n}$.
\end{itemize}
Thus, we see that 
\begin{equation}\label{Eq:0303}
\sum_{B \in \B_{d,n}} v(B) = -\mu_d(n+1).
\end{equation}

Next, for every $d \geq 1$ and $n \geq 0$, let $\A_{d,n}$ to be the set of sequences $(A_1, \ldots, A_k)$ where
\begin{itemize}
\item
$A_i$ is a $d$-coloured prime set for every $i \in [k]$;
\item
$\sum_{i=1}^k w(A_i) = n$.
\end{itemize}
Also, we define the sign of $A$ to be $v(A) = \prod_{i=1}^k v(A_i)$. In other words, $v(A) = 1$ if there is an odd number of even-sized sets in the sequence $A$, and $v(A) = -1$ otherwise. For example, the following table lists all elements of $A_{1,n}$ for $0 \leq n \leq 7$, as well as their signs. To reduce cluttering, we used $|$ to indicate separation of sets, and suppressed the colour assignment of the numbers since there is only $d=1$ available colour in this case.
\[
\begin{array}{r|r|l|l}
n & a_1(n) & \set{A \in \A_{1,n} : v(A) = 1} & \set{A \in \A_{1,n} : v(A) = -1}\\
\hline
0 & 1 & () & \\
\hline
1 & 1 & (2)& \\
\hline
2 & 2 & (2|2), (3)& \\
\hline
3 & 3 & (2|2|2), (2|3), (3|2)& \\
\hline
4 & 6 & (2|2|2|2), (2|2|3), (2|3|2), (3|2|2), (3|3), (5)& \\
\hline
5 & 9 & (2|2|2|2|2), (2|2|2|3), (2|2|3|2), (2|3|2|2), (2|3|3), & (2,3) \\
&& 
(2|5),(3|2|2|2), (3|2|3), (3|3|2), (5|2) 
 & \\
\hline
6 & 17 &
(2|2|2|2|2|2), (2|2|2|2|3), (2|2|2|3|2), (2|2|3|2|2), &(2|2,3), (2,3|2) \\
&& 
(2|2|3|3), (2|2|5), (2|3|2|2|2), (2|3|2|3), (2|3|3|2), & \\
&& 
(2|5|2), (3|2|2|2|2), (3|2|2|3), (3|2|3|2), (3|3|2|2),& \\
&& 
 (3|3|3), (3|5), (5|2|2), (5|3), (7)  & 
\end{array}
\]
Then we have the following.

\begin{lemma}\label{Lem:0301}
For every $d \geq 1$ and $n \geq 0$, 
\[
a_d(n) = \sum_{A \in \A_{d,n}} v(A).
\]
\end{lemma}

\begin{proof}
We prove the claim by induction on $n$. First, $a_d(0) = 1$, and $\A_{d,0}$ consists of just the empty sequence $A = ()$, and $v(A) = 1$. Thus, the base case holds.

Now suppose $n \geq 1$, and let $A = (A_1, \ldots, A_k) \in \A_{d,n}$. Then $A_1 \in \B_{d,i}$ for some $1 \leq i \leq n-1$, and $A' = (A_2, \ldots, A_k) \in \A_{d,n-i}$. Thus, there is a natural bijection $g : \A_{d,n} \to \bigcup_{i=1}^{n-1} \B_{d,i} \times \A_{d, n-i}$. Moreover, notice that $v(A) = v(A_1)v(A')$. Thus, 

\begin{align*}
\sum_{A \in \A_{d,n}} v(A) 
&= \sum_{i=1}^{n-1} \left( \sum_{B \in \B_{d,i}} v(B) \sum_{A' \in \A_{d,n-i}}  v(A') \right) \\
& \overset{\eqref{Eq:0303}}{=} \sum_{i=1}^{n-1} - \mu_d(n+1) a_d(n-i) \\
&= \sum_{k=2}^n -\mu_d(k) a_d(n+1-k) \\
&\overset{\eqref{Eq:0302}}{=} a_d(n). \mbox{\qedhere}
\end{align*}
\end{proof}

Next, we show that there are always more sequences in $\A_{d,n}$ with positive signs than those with negative signs. Given a sequence $A = (A_1, \ldots, A_k) \in \A_{d,n}$, we say that a subsequence $(A_j, A_{j+1}, \ldots, A_{j+\l})$ of $A$ is \emph{odd, ascending, and repetitive} (OAR) if
\begin{enumerate}
\item
(Odd) $A_j= \set{\l}$, a singleton set, and $|A_{j+1}|, \ldots, |A_{j+\l}|$ are all odd.
\item
(Ascending) Let $c_0$ be the colour assigned to the element $\l \in A_j$. Then every element of $A_{j+1}$ is either greater than $\l$, or is equal to $\l$ and assigned a colour greater than $c_0$.
\item
(Repetitive) $A_{j+1} = A_{j+2} = \cdots = A_{j+\l}$.
\end{enumerate}
For example, $A=(\set{3}, \set{2}, \set{2},\set{3,5,11},\set{3,5,11}, \set{2,3})$ contains an OAR subsequence starting at $j =3$ with $\ell = 2$ (regardless of $d$ and the colour assigned to the elements). Then we have the following.

\begin{lemma}\label{Lem:0302}
For every $d,n \geq 1$,
\[
|\set{A \in \A_{d,n} : v(A) = -1} | \leq |\set{A \in \A_{d,n} : v(A) = 1}|.
\]
\end{lemma}

\begin{proof}
Let $A = (A_1, \ldots, A_k) \in \A_{d,n}$. Define indices $i_1(A), i_2(A)$ as follows:
\begin{itemize}
\item
$i_1(A)$ is the smallest index where $|A_i|$ is even; 
\item
$i_2(A)$ is the smallest index where an OAR subsequence begins.
\end{itemize}
Notice that, when both defined, $i_1(A)$ and $i_2(A)$ must be distinct for any sequence $A$, since an OAR subsequence must begin with a singleton set, which has odd size. Next, given sequence $A$, we define the sequence $f(A)$ as follows:
\begin{itemize}
\item
Case 1:  $i_1(A)$ is defined, and $i_2(A) > i_1(A)$ or is undefined. Let $i = i_1(A)$, and let $p_0 \in A_i$ be the smallest prime number that is assigned the minimum colour. Let $A'_{i} = A_i \setminus \set{p_0}$. Note that, since $|A_i|$ is even, $|A'_i|$ is necessarily odd (and hence nonempty). Define $f(A)$ by taking the sequence $A$ and replacing $A_i$ by the (OAR) subsequence
\[
\set{p_0}, \underbrace{A', \ldots, A'}_{\tn{$p_0$~times}},
\]
with the colour assignment to elements unchanged. 

For example, given $A = (\set{2}, \set{3,11}, \set{5,7})$, $i_1(A) =2$ and $i_2(A)$ is undefined, and we have $f(A) = (\set{2}, \set{3}, \set{11}, \set{11}, \set{11}, \set{5,7})$. Notice that $v(A) = 1, v(f(A))=-1$, and both $A$ and $f(A)$ have weight $1+32+34 = 67$.
\item
Case 2: $i_2(A)$ is defined, and $i_1(A) > i_2(A)$ or is undefined. Let $i = i_2(A)$ and let $p_0$ be the unique element in $A_i$. Define $f(A)$ by taking $A$ and replacing the OAR subsequence $A_j, A_{j+1}, \ldots, A_{j+p_0}$ by the set $A_j \cup A_{j+1}$, with the colour assignment to elements unchanged.

For example, given $A = (\set{2}, \set{3}, \set{11}, \set{11}, \set{11}, \set{5,7})$, then $i_1(A) = 6$ and $i_2(A) = 2$, and we have $f(A) = (\set{2}, \set{3,11}, \set{5,7})$.
\end{itemize}
Notice that if $v(A) = -1$, then $i_1(A)$ must be defined, and so $f(A)$ is defined. Since $f$ either replaces one even set by a number of odd sets (Case 1) or vice versa (Case 2), we see that $v(f(A)) = -v(A)$ whenever $f(A)$ is defined. It is also not hard to check that $w(f(A)) = w(A)$ in both cases.

Furthermore, notice that if $f(A)$ is defined, then $f(f(A)) = A$. This shows that $f$ is a bijection between the sets $\set{A \in \A_{d,n} : v(A) = -1}$ and $\set{f(A) : A \in \A_{d,n}, v(A) = -1}$. Thus, $f$ is injective when we consider it as a mapping from $\set{A \in \A_{d,n} : v(A) = -1}$ to $|\set{A \in \A_{d,n} : v(A) = 1}$, and our claim follows.
\end{proof}

Thus, the number of elements $A \in \A_{d,n}$ with $v(A)=-1$ never outnumber those with $v(A) = 1$. In fact, the function $f$ defined in the proof above can be seem as ``pairing up'' the sequences in $\A_{d,n}$ that contains an even set or an OAR subsequence (or both). Thus, we have obtained the following.

\begin{corollary}\label{Cor:0301}
Define $\tilde{\A}_{d,n} \subseteq \A_{d,n}$ to be the set of sequences that neither contains an even set nor an OAR subsequence. Then $a_d(n) = |\tilde{\A}_{d,n}|$ for every $d \geq 1$ and $n \geq 0$. 
\end{corollary}

Corollary~\ref{Cor:0301} gives us a set whose elements is counted by $a_d(n)$. Using this combinatorial description, we prove the following result.

\begin{proposition}\label{Prop:0301}
For every $d \geq 1$, $a_d(0) =1$ and
\[
\frac{a_d(n+1)}{a_d(n)} \geq d
\]
for all $n \geq 0$.
\end{proposition}

\begin{proof}
Given $A = (A_1, \ldots, A_k) \in \tilde{\A}_{d,n}$, we define $f : \tilde{\A}_{d,n} \times [d] \to \tilde{\A}_{d, n+1}$ as follows:
\[
f(A, c) = \begin{cases}
(A_1, \ldots, A_{k-1}, \set{3_c}) & \tn{ if $A_{k} = \set{2_c}$ and $A_{k-1} = \set{2_{c'}}$ where $c' < c$;}\\
(A_1, \ldots, A_{k-1}, A_k, \set{2_c}) & \tn{otherwise.}
\end{cases}
\] 
(Again, we have used subscripts to denote assigned colours of elements.) The function $f$ can be seen as appending the set $\set{2_c}$ at the end of the given sequence (which increases the total weight by $1$ and does not create an even set). If this does create an OAR subsequence, it must be that the last three sets of the new sequence are $\set{2_{c'}}, \set{2_c}, \set{2_c}$ where $c' < c$. In this case, we further replace the two instances of $\set{2_c}$ by one instance of $\set{3_c}$, which does not change the weight of the whole sequence, and now guarantees that it does not have an OAR subsequence.

Since $f(A,c)$ must have weight $n+1$ and does not contain any even sets nor OAR subsequences, it is an element of $\tilde{\A}_{d,n+1}$. It is also easy to see that $A$ and $c$ are uniquely recoverable from $f(A,c)$, and so $f(A,c)$ is injective. Thus, we conclude that 
\[
a_d(n+1) = |\tilde{\A}_{d,n+1}| \geq | \tilde{\A}_{d,n} \times [d]| = d a_d(n). \mbox{\qedhere}
\]
\end{proof}

\subsection{An upper bound for $a_d(n+1)/a_d(n)$}

\begin{figure}[h!]
\begin{center}
\begin{tikzpicture}[scale = 0.65, xscale = 1, yscale = 1, font = \footnotesize, word node/.style={font=\footnotesize}]]
\def\xlb{0};
\def\xub{20};
\def\ylb{0};
\def\yub{7};
\def\xbuf{0.5};
\def\ybuf{0.5};
\draw [->] (0,0) -- (\xub + \xbuf,0);
\draw [->] (0,0) -- (0, \yub + \ybuf);
\foreach \x in {\xlb ,5, ...,\xub}
{
\ifthenelse{\NOT 0 = \x}{\draw[thick](\x ,-2pt) -- (\x ,2pt);}{}
\ifthenelse{\NOT 0 = \x}{\node[anchor=north] at (\x,0) (label) {{$\x$}};}{}
}
\foreach \y in {1,  ..., \yub}
{
\draw[thick](-2pt, \y ) -- (2pt, \y);
\draw[dotted](0, \y ) -- (\xub + \xbuf, \y);
\node[anchor=east] at (0,\y) (label) {{ $\y$}};
}
\node[anchor=north east] at (0,0) (label3) { $0$};
\node[anchor=west] at (\xub + \xbuf,0) (label3) {$n$};
\node[anchor=south] at (0, \yub+\ybuf) (label3) {$\frac{a_d(n+1)}{a_d(n)}$};
%d=1
\draw[thick](0,1)--(1,2)--(2,1.5)--(3,2)--(4,1.5)--(5,1.888888888888889)--(6,1.647058823529412)--(7,1.785714285714286)--(8,1.66)--(9,1.771084337349398)--(10,1.693877551020408)--(11,1.746987951807229)--(12,1.705747126436782)--(13,1.735849056603774)--(14,1.713509316770186)--(15,1.730403262347078)--(16,1.717989002356638)--(17,1.727328151196464)--(18,1.720374128650843)--(19,1.725393650305175)--(20,1.721611177170036);
\def\y{0.5};
\foreach \position in {(0,1),(1,2),(2,1.5),(3,2),(4,1.5),(5,1.888888888888889),(6,1.647058823529412),(7,1.785714285714286),(8,1.66),(9,1.771084337349398),(10,1.693877551020408),(11,1.746987951807229),(12,1.705747126436782),(13,1.735849056603774),(14,1.713509316770186),(15,1.730403262347078),(16,1.717989002356638),(17,1.727328151196464),(18,1.720374128650843),(19,1.725393650305175),(20,1.721611177170036)}
{\node[draw, circle, inner sep=0pt, minimum size = 0.1cm, fill] at \position {};}
\node[anchor = west] at (20,  1.721611177170036) {$d=1$};
%d=2
\draw[thick](0,2)--(1,3)--(2,2.5)--(3,2.8)--(4,2.571428571428572)--(5,2.694444444444445)--(6,2.632302405498282)--(7,2.664490861618799)--(8,2.643312101910828)--(9,2.655792400370713)--(10,2.649078726968174)--(11,2.652782168827063)--(12,2.650597384024074)--(13,2.651855099180924)--(14,2.651141728000136)--(15,2.651541548994392)--(16,2.651313455509663)--(17,2.651444812734057)--(18,2.65137003380384)--(19,2.651412286573063)--(20,2.651388201844219);
\def\y{0.5};
\foreach \position in {(0,2),(1,3),(2,2.5),(3,2.8),(4,2.571428571428572),(5,2.694444444444445),(6,2.632302405498282),(7,2.664490861618799),(8,2.643312101910828),(9,2.655792400370713),(10,2.649078726968174),(11,2.652782168827063),(12,2.650597384024074),(13,2.651855099180924),(14,2.651141728000136),(15,2.651541548994392),(16,2.651313455509663),(17,2.651444812734057),(18,2.65137003380384),(19,2.651412286573063),(20,2.651388201844219)}
{\node[draw, circle, inner sep=0pt, minimum size = 0.1cm, fill] at \position {};}
\node[anchor = west] at (20,2.651388201844219) {$d=2$};
%d=3
\draw[thick](0,3.000000)--(1,4.000000)--(2,3.500000)--(3,3.714286)--(4,3.576923)--(5,3.63
4409)--(6,3.608481)--(7,3.619841)--(8,3.613817)--(9,3.616609)--(10,3.615323)-
-(11,3.615914)--(12,3.615628)--(13,3.615763)--(14,3.615699)--(15,3.615729)--(
16,3.615715)--(17,3.615721)--(18,3.615718)--(19,3.615720)--(20,3.615719);
\def\y{0.5};
\foreach \position in {(0,3.000000),(1,4.000000),(2,3.500000),(3,3.714286),(4,3.576923),(5,3.634409)
,(6,3.608481),(7,3.619841),(8,3.613817),(9,3.616609),(10,3.615323),(11,3.6159
14),(12,3.615628),(13,3.615763),(14,3.615699),(15,3.615729),(16,3.615715),(17
,3.615721),(18,3.615718),(19,3.615720),(20,3.615719)}
{\node[draw, circle, inner sep=0pt, minimum size = 0.1cm, fill] at \position {};}
\node[anchor = west] at (20,3.615719) {$d=3$};
%d=4
\draw[thick](0,4.000000)--(1,5.000000)--(2,4.500000)--(3,4.666667)--(4,4.571429)--(5,4.60
4167)--(6,4.590498)--(7,4.595614)--(8,4.593276)--(9,4.594186)--(10,4.593813)-
-(11,4.593961)--(12,4.593899)--(13,4.593924)--(14,4.593914)--(15,4.593918)--(
16,4.593916)--(17,4.593917)--(18,4.593916)--(19,4.593916)--(20,4.593916);
\def\y{0.5};
\foreach \position in {(0,4.000000),(1,5.000000),(2,4.500000),(3,4.666667),(4,4.571429),(5,4.604167)
,(6,4.590498),(7,4.595614),(8,4.593276),(9,4.594186),(10,4.593813),(11,4.5939
61),(12,4.593899),(13,4.593924),(14,4.593914),(15,4.593918),(16,4.593916),(17
,4.593917),(18,4.593916),(19,4.593916),(20,4.593916)}
{\node[draw, circle, inner sep=0pt, minimum size = 0.1cm, fill] at \position {};}
\node[anchor = west] at (20,4.593916) {$d=4$};
%d=5
\draw[thick](0,5.000000)--(1,6.000000)--(2,5.500000)--(3,5.636364)--(4,5.564516)--(5,5.58
5507)--(6,5.577236)--(7,5.579927)--(8,5.578824)--(9,5.579197)--(10,5.579058)-
-(11,5.579106)--(12,5.579088)--(13,5.579095)--(14,5.579092)--(15,5.579093)--(
16,5.579093)--(17,5.579093)--(18,5.579093)--(19,5.579093)--(20,5.579093);
\def\y{0.5};
\foreach \position in {(0,5.000000),(1,6.000000),(2,5.500000),(3,5.636364),(4,5.564516),(5,5.585507)
,(6,5.577236),(7,5.579927),(8,5.578824),(9,5.579197),(10,5.579058),(11,5.5791
06),(12,5.579088),(13,5.579095),(14,5.579092),(15,5.579093),(16,5.579093),(17
,5.579093),(18,5.579093),(19,5.579093),(20,5.579093)}
{\node[draw, circle, inner sep=0pt, minimum size = 0.1cm, fill] at \position {};}
\node[anchor = west] at (20,5.579093) {$d=5$};
%d=6
\draw[thick](0,6.000000)--(1,7.000000)--(2,6.500000)--(3,6.615385)--(4,6.558140)--(5,6.57
2695)--(6,6.567215)--(7,6.568790)--(8,6.568196)--(9,6.568375)--(10,6.568313)-
-(11,6.568332)--(12,6.568325)--(13,6.568327)--(14,6.568327)--(15,6.568327)--(
16,6.568327)--(17,6.568327)--(18,6.568327)--(19,6.568327)--(20,6.568327);
\def\y{0.5};
\foreach \position in {(0,6.000000),(1,7.000000),(2,6.500000),(3,6.615385),(4,6.558140),(5,6.572695)
,(6,6.567215),(7,6.568790),(8,6.568196),(9,6.568375),(10,6.568313),(11,6.5683
32),(12,6.568325),(13,6.568327),(14,6.568327),(15,6.568327),(16,6.568327),(17
,6.568327),(18,6.568327),(19,6.568327),(20,6.568327)}
{\node[draw, circle, inner sep=0pt, minimum size = 0.1cm, fill] at \position {};}
\node[anchor = west] at (20,6.568327) {$d=6$};
\end{tikzpicture}
\end{center}
\caption{Plotting $\frac{a_d(n+1)}{a_d(n)}$ for $1 \leq d \leq 6, 0 \leq n \leq 20$.}
\label{Fig:0301}
\end{figure}
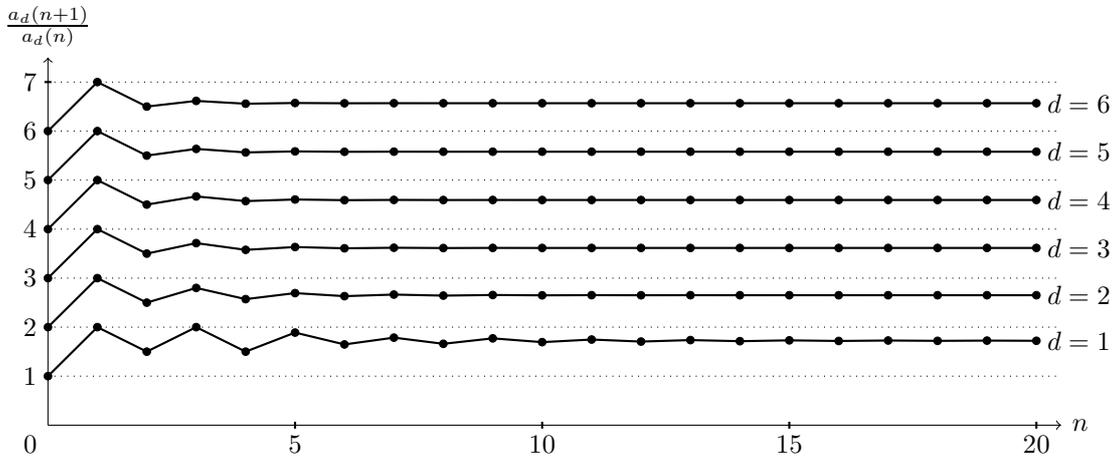

We illustrate in Figure~\ref{Fig:0301} the values of $\frac{a_d(n+1)}{a_d(n)}$ for some small values of $d$ and $n$. While we have shown that the ratio is bounded below by $d$, the figure suggests that $d+1$ is a tight upper bound. We provide an algebraic proof that this is indeed true for all $d \geq 3$.

\begin{proposition}\label{Prop:0302}
For every $d \geq 3$ and $n \geq 0$, 
\[
\frac{a_d(n+1)}{a_d(n)} \leq d+1.
\]
\end{proposition}

To do so, we'll need a lemma that will also be useful in Section~\ref{Sec:04}.

\begin{lemma}\label{Lem:0303}
Let $d \geq 2$ and $k \geq 3$ be integers.
\begin{itemize}
\item[(i)]
For all $x \in [0, \frac{1}{d}]$, 
\[
\left| \sum_{n\geq 2^k} \mu_d(n)x^n  \right| \leq  2d^k x^{2^k}.
\]
\item[(ii)]
For all $x \in [0, \frac{1}{2d}]$, 
\[
\left| \sum_{n\geq 2^k} n \mu_d(n)x^{n-1}  \right| \leq  2^{k+1} d^k x^{2^k-1}.
\]
\end{itemize}
\end{lemma}

\begin{proof}
We first prove $(i)$. Notice that for all $n \in \set{ 2^k, \ldots, 2^{k+1}-1}$, $n$ has at most $k$ prime factors (counting multiplicities), and so $|\mu_d(n) | \leq d^k$. Thus,
\begin{align*}
\left| \sum_{n \geq 2^k} \mu_d(n) x^n \right| 
&= \left| \sum_{\ell \geq k} \sum_{n = 2^{\ell}}^{2^{\ell+1} -1} \mu_d(n) x^n \right| \\
&\leq \sum_{\ell \geq k} \sum_{n = 2^{\ell}}^{2^{\ell+1} -1} d^{\ell} x^n  \\
&\leq \sum_{\ell \geq k} \sum_{n \geq 2^{\ell}} d^{\ell} x^n\\
&=  \sum_{\ell \geq k} \frac{ d^{\ell}x^{2^{\ell}}}{1-x} \\
&\leq  \sum_{\ell \geq k} \frac{ d^{\ell}x^{ 2^{k} (\ell-k+1)}}{1-x} \\
&=  \frac{ d^k x^{2^k} }{(1-x) (1- dx^{2^k}) }.
\end{align*}
For the last inequality, notice that $2^{i-1} \geq i$ for all $i \geq 1$. Substituting $i = \ell-k+1$ and then multiplying both sides by $2^k$ gives $2^{\ell} \geq 2^{k}(\ell-k+1)$. Since $0 \leq x \leq \frac{1}{d} < 1$, it follows that $x^{2^{\ell}} \leq x^{2^{\ell}(\ell-k+1)}$ for all $\ell \geq k$.

Next, when $d \geq 3, k \geq 3$, and $x \leq \frac{1}{3}$, $(1-x)(1-dx^{2^k}) \geq \frac{1}{2}$, and so $\frac{ d^k x^{2^k} }{(1-x) (1- dx^{2^k})} \leq 2d^kx^{2^k}$. For the case $d =2$, notice that $\mu_2(2^k) = 0$ for all $k\geq 3$. Thus, using essentially the same chain of inequalities above, we obtain that
\[
\left| \sum_{n \geq 2^k} \mu_2(n) x^n \right| \leq \frac{ d^k x^{2^k} }{(1-x) (1- dx^{2^k}) } -d^kx^{2^k} \leq 2 d^kx^{2^k}. 
\]
Next, we prove $(ii)$ using similar observations. Given $d\geq 2, k\geq 3$, and $x \leq \frac{1}{2d}$, 
\begin{align*}
\left| \sum_{n \geq 2^k} \mu_d(n) n x^{n-1} \right| 
&= \left| \sum_{\ell \geq k} \sum_{n = 2^{\ell}}^{2^{\ell+1} -1} \mu_d(n) nx^{n-1} \right| \\
&\leq \sum_{\ell \geq k} \sum_{n = 2^{\ell}}^{2^{\ell+1} -1} d^{\ell} nx^{n-1}  \\
&\leq \sum_{\ell \geq k} \sum_{n \geq 2^{\ell}} d^{\ell}  nx^{n-1} \\
&= \sum_{\ell \geq k}  d^{\ell}x^{2^{\ell}-1}  \left( \frac{ 2^{\ell}}{1-x} + \frac{x}{(1-x)^2} \right) \\
&\leq \sum_{\ell \geq k}  d^{\ell} x^{(2^k-1)+(\ell -k) 2^k}  \left( \frac{ 2^{\ell}}{1-x} + \frac{x}{(1-x)^2} \right) \\
&= d^{\ell} x^{2^k-1}  \left( \frac{ 2^k}{(1-2dx^{2^k})(1-x)} + \frac{x}{(1-dx^{2^k})(1-x)^2} \right).\\
& \leq d^{\ell} x^{2^k-1} \left( \frac{3}{2} 2^k + \frac{1}{2} \right)\\
& \leq 2^{k+1} d^{\ell} x^{2^k-1}. \mbox{\qedhere}
\end{align*}
\end{proof}

We are now ready to prove Proposition~\ref{Prop:0302}.

\begin{proof}[Proof of Proposition~\ref{Prop:0302}]
We prove our claim by induction on $n$. First, using~\eqref{Eq:0302}, we obtain that
\begin{align*}
a_d(0) &= 1,\\
a_d(1) &= d,\\
a_d(2) &= d^2+d,\\
a_d(3) &= d^3+\frac{3}{2}d^2+\frac{1}{2}d,\\
a_d(4) &= d^4+2d^3+2d^2+d,\\
a_d(5) &= d^5+\frac{5}{2}d^4 + \frac{7}{2}d^3 + 2d^2,\\
a_d(6) &= d^6+3d^5+\frac{21}{4}d^4+\frac{9}{2}d^3+\frac{9}{4}d^2+d.
\end{align*}
From the above, one can check that $\frac{a_d(n+1)}{a_d(n)} \leq d+1$ for all $0 \leq n \leq 5$. Next, assume $n \geq 6$. Observe that
\begin{align*}
a_d(n) ={}& da_d(n-1) + da_d(n-2) - \binom{d}{2}a_d(n-3) + da_d(n-4) - d^2a_d(n-5)  \\
&+ da_d(n-6) - \sum_{k \geq 8}^{n+1} \mu_d(k) a_d(n+1-k).
\end{align*}
Using the inductive hypothesis as well as Proposition~\ref{Prop:0301}, we see that
\begin{align*}
 da_d(n-2) &\leq a_d(n-1), \\
-\binom{d}{2} a_d(n-3) & \leq -d^3 \binom{d}{2} a_d(n-6),\\
da_d(n-4) & \leq  d(d+1)^2 a_d(n-6),\\
-d^2a_d(n-5) & \leq -d^3a_d(n-6).
\end{align*}
Also, from Proposition~\ref{Prop:0301} again, $a_d(n-6-i) \leq d^{-i}a_d(n-6) $ for every $i\geq 1$, and so
\begin{align*}
\left| \sum_{k \geq 8}^{n+1} \mu_d(k) a_d(n+1-k)\right| & \leq \left| \sum_{k \geq 8}^{n+1} \mu_d(k) d^{7-k} a_d(n-6) \right|  \leq \left| \sum_{k \geq 8} \mu_d(k) d^{7-k} a_d(n-6) \right| \\
& \leq 2d^2 a_d(n-6)
\end{align*}
using Lemma~\ref{Lem:0303}(i). Thus,
\begin{align*}
a_d(n) ={}& da_d(n-1) +  da_d(n-2)   - \binom{d}{2}a_d(n-3)+ da_d(n-4) - d^2a_d(n-5) \\
& + da_d(n-6)  - \sum_{k \geq 8}^{n+1}  \mu_d(k) a_d(n+1-k) \\
\leq{}& (d+1)a_d(n-1) + \left( -d^3\binom{d}{2} + d(d+1)^2 - d^3 + d + 2d^2\right) a_d(n-6)\\
\leq{}& (d+1)a_d(n-1),
\end{align*}
where the last inequality relies on the assumption that $d \geq 3$.
\end{proof}

As seen in Figure~\ref{Fig:0301}, it appears that Proposition~\ref{Prop:0302} is also true for $d=1$ and $d=2$. A combinatorial proof for that would be interesting.

\section{Asymptotic Formula and Growth Rate of $s_d(n)$}\label{Sec:04}

In this section, we consider the asymptotic behavior and growth rate of $s_d(n)$. We will also end the section by describing a mapping from labelled plane rooted trees to hypercube decompositions that will help provide additional context to the growth rate of $s_d(n)$.

\subsection{An asymptotic formula}

The main tool we will rely on is the following result from Flajolet and Sedgewick~\cite[Theorem VI.6, p.~404]{FlajoletS09}:

\begin{theorem}\label{Thm:0401}
Let $y$ be a power series in $x$.
Let $\phi\colon \mathbb{C} \to \mathbb{C}$ be a function with the following properties:
\begin{enumerate}
\item[(i)]
$\phi$ is analytic at $z = 0$ and $\phi(0) > 0$;
\item[(ii)]
$y = x \phi(y)$;
\item[(iii)]
$[z^n] \phi(z) \geq 0$ for all $n \geq 0$, and $[z^n] \phi(z) \neq 0$ for some $n \geq 2$.
\item[(iv)]
There exists a (then necessarily unique) real number $s \in (0,r)$ such that
$\phi(s) = s \phi'(s)$, where $r$ is the radius of convergence of $\phi$.
\end{enumerate}
Then, 
\[
[x^n] y \sim \sqrt{ \frac{\phi(s)}{2\pi  \phi''(s)}} \, n^{-3/2} \left(  \phi'(s) \right)^n.
\]
\end{theorem}

Using Theorem~\ref{Thm:0401}, we obtain the following:

\begin{theorem}\label{Thm:0402}
Let $d \geq 1$, and let $s > 0$ be the smallest real number such that $M_d'(s) = 0$. Then,
\[
s_d(n) \sim \frac{1}{ \sqrt{ -2\pi M_d''(s)}} \, n^{-3/2} \, M_d(s)^{\tfrac12-n}.
\]
\end{theorem}

\begin{proof}
We first verify the analytic conditions listed in Theorem~\ref{Thm:0401}. Since $M_d(y) = x$, to satisfy $(ii)$ we have $y=x\phi(y)$ where 
\[
\phi(z) = \frac{z}{M_d(z)} =  \sum_{n \geq 0} a_d(n)z^n,
\]
where the coefficients $a_d(n)$ were defined in~\eqref{Eq:0301} and studied in Section~\ref{Sec:03}. It is easy to see that $\phi(0) = 1$ (for all $d$) and that $\Phi$ is analytic at $z=0$, and so $(i)$ holds. Condition $(iii)$ follows readily from Proposition~\ref{Prop:0301}. For $(iv)$, notice that
\begin{equation}\label{Eq:0401}
\phi'(z) = \frac{M_d(z) - zM_d'(z)}{M_d(z)^2},
\quad
\phi''(z) = \frac{-zM_d(z)M_d''(z) - 2M_d(z)M_d'(z) + 2z(M_d'(z))^2}{M_d(z)^3}.
\end{equation}
Let $r$ be the radius of convergence of $\phi$ at $z = 0$. 
Since $\phi(z) = \frac{z}{M_d(z)}$, $r$ is the smallest positive solution to $M_d(r) = 0$. Given $M_d(0) = M_d(r)=0$ and that $M_d(z)$ is differentiable over $(0,r)$, there must exist $s \in (0,r)$ where $M_d'(s) =0$. Now observe that
\[
M_d'(s) = 0
~\implies~
\frac{s}{M_d(s)} = s \left( \frac{M_d(s) - sM_d'(s)}{M_d(s)^2} \right) 
~\implies~
\phi(s) = s \phi'(s).
\]
Thus, condition $(iv)$ holds. Now that the analytic assumptions on $\phi(z)$ have been verified, 
we may establish the asymptotic formula. 
When $M_d'(s) = 0$, the expressions in~\eqref{Eq:0401} simplifies to
\[
\phi'(s) = \frac{1}{M_d(s)}, 
\quad\quad\quad 
\phi''(s) = \frac{-sM_d''(s)}{M_d(s)^2}.
\]
Therefore,  we have
\begin{align*}
s_d(n) = [x^n] y 
& \sim
 \sqrt{ \frac{\phi(s)}{2 \pi \phi''(s)} } \; n^{-3/2} \, \phi'(s)^n 
\\
&= \sqrt{ 
\frac{ 
s / M_d(s)
}
{
2 \pi \big( -s M_d''(s) / M_d(s)^2 \big)
}
} 
\; n^{-3/2}  \left( \frac{1}{M_d(s)} \right)^n 
\\
&\;=\; 
\frac{1}{ \sqrt{ -2\pi M_d''(s)  }} \, n^{-3/2} \, M_d(s)^{1/2-n}.
\end{align*}
This completes the proof.
\end{proof}

\subsection{Growth rate}

We define the growth rate of $s_d(n)$ to be
\[
\K_d = 
\lim_{n \to \infty} 
\frac{s_d(n+1)}{ s_d(n) }.
\]
Goulden et al.~\cite[Theorem 2]{GouldenGRS18} showed that $\K_1 \approx 5.487452$ in their work on NECS. Here, we show that $d=1$ turns out the only case where $\K_d < 4d + \frac{3}{2}$. 

\begin{proposition}\label{Prop:0401}
For all integers $d \geq 2$,
\[
4d+ \frac{3}{2}  \leq \K_d \leq 4d + \frac{3}{2} + \frac{1}{16d}.
\].
\end{proposition}

\begin{proof}
It follows immediately from Theorem \ref{Thm:0402} that $\K_d = \frac{1}{M_d(s)}$
where $s$ is the smallest positive real number where $M_d'(s) =0$. For convenience, we define the polynomials
\[
M_d^-(x) = \sum_{n \geq 1}^7 \mu_d(n)x^n - 2d^3x^8, \quad 
M_d^+(x) = \sum_{n \geq 1}^7 \mu_d(n)x^n + 2d^3x^8.
\]
Then from Lemma~\ref{Lem:0303} we know that $M_d^-(x) \leq M_d(x) \leq M_d^+(x)$ over $[0, \frac{1}{d}]$, and $(M_d^-)'(x) \leq M_d'(x) \leq (M_d^+)'(x)$ over $[0, \frac{1}{2d}]$.

Let $k_1 = \frac{4d+5}{(4d+5)(2d+1) +1}$ and $k_2 = \frac{d-1}{d} k_1 + \frac{1}{d(2d+1)}$. Notice that $0 < k_1 < k_2 < \frac{1}{2d}$. Now observe that $(M_d^+)'(k_1) = \frac{1}{64(4d^2+7d+3)^7}c_1(d)$, where
\begin{align*}
c_1(d) ={}& 229376d^{10}+1966080d^9+7294976d^8+15323136d^7+20124288d^6+17555072d^5\\
&+11102496d^4+5917032d^3+2803847d^2+933639d+139968.
\end{align*}
(All polynomials computations in this proof were performed in Maple.) Likewise, one can check that $(M_d^+)'(k_2) = \frac{1}{64d^5(2d+1)^7(4d^2+7d+3)^7} c_2(d)$ where
\begin{align*}
c_2(d) ={}&
-16777216d^{23}-209715200d^{22}-1265631232d^{21}-4968939520d^{20}\\
& 
-14345764864d^{19}-32330973184d^{18}-58457362432d^{17}-85738966016d^{16}\\
& 
-102477934592d^{15}-100110810240d^{14}-80143714368d^{13}-52681506144d^{12}\\
& 
-28456978128d^{11}-12623832456d^{10}-4590955884d^9-1365827366d^8\\
& 
-331953575d^7-65949629d^6-10735783d^5-1431097d^4-153709d^3-12647d^2\\
& 
-713d-21.
\end{align*}
Since $c_1(d)$ has exclusively positive coefficients, $(M_d^+)'(k_1) > 0$. Similarly, we see that $(M_d^+)'(k_2) < 0$. Thus, $M_d^+(x)$ has a local maximum at $s^+ \in (k_1, k_2)$. Next, observe that
\[
(M_d^+)''(x) = -2d - 6dx + 6d(d-1)x^2 - 20dx^3 + 30d^2x^4 - 42dx^5 + 112d^3x^6,
\]
which is negative over $[0, \frac{1}{2d}]$. Thus, $M_d^+(x)$ is concave down over this interval. and $M_d^+(s^+)$ is the absolute maximum of $M_d^+(x)$ over $[0, \frac{1}{2d}]$. Now consider
\[
L(x) = M_d^+(k_1) + (M_d^+)'(k_1)(x-k_1),
\]
the linearization of $M_d^+(x)$ at $x=  k_1$. We argue that
\[
M_d(s) \leq M_d^+(s) \leq M_d^+(s^+) \leq L(s^+) \leq L(k_2).
\]
For the first inequality, notice that $M_d'(0) > 0$, and $M_d'(k_2) < (M_d^+)'(k_2) < 0$. Thus, $s \in [0, \frac{1}{2d}]$, which implies that $M_d(s) \leq M_d^+(s)$ (since $M_d(x) \leq M_d^+(x)$ over $[0, k_2]$). The second inequality holds because we showed above that $s^+$ maximizes $M_d^+(x)$ over $[0, \frac{1}{2d}]$. The third inequality holds since $M_d^+(x)$ is concave down, and so $L(x) \geq M_d^+(x)$ over $[k_1, k_2]$. Finally, since $(M_d^+)'(k_1) >0$, $L(x)$ is an increasing function, and we obtain that $L(s^+) \leq L(k_2)$. 

Thus, we have $M_d(s) \leq  L(k_2) = M_d^+(k_1) + (k_2-k_1) (M_d^+)'(k_1)$, and so
\[
\K_d = \frac{1}{M_d(s)}
\geq \frac{1}{M_d^+(k_1) + (k_2-k_1) (M_d^+)'(k_1)}
= 4d+\frac{3}{2} + \frac{c_3(d)}{c_4(d)},
\]
where
\begin{align*}
c_3(d) ={}&
524288d^{16}+6029312d^{15}+29491200d^{14}+76218368d^{13}+92368896d^{12}\\
&
-45438976d^{11}-403496448d^{10}-849085696d^9-1125540408d^8-1083739620d^7\\
&
-793966088d^6-450174888d^5-201247133d^4-73160887d^3-21407400d^2\\
&
-4340565d-419904,\\
c_4(d)={}&
8388608d^{17}+118489088d^{16}+783810560d^{15}+3223584768d^{14}+9227272192d^{13}\\
&
+19497541632d^{12}+31470530560d^{11}+39594894336d^{10}+39261180544d^9\\
&
+30802044560d^8+19075944984d^7+9245407152d^6+3450501536d^5\\
&
+966637518d^4+196684154d^3+28701088d^2+3266958d+279936.
\end{align*}
One can check that $c_3(d), c_4(d) \geq 0$ for all $d \geq 2$, and so we obtain that $\K_d  \geq 4d + \frac{3}{2}$.

Next, we prove the upper bound. Since $M_d(s) \geq M_d(k_1) \geq M_d^-(k_1)$, we have
\[
\K_d = \frac{1}{M_d(s)} 
 \leq \frac{1}{M_d^-(k_1)}
 = 4d + \frac{3}{2} + \frac{16}{d} - \frac{16}{d} \cdot \frac{c_5(d)}{c_6(d)},
\]
where
\begin{align*}
c_5(d)={}&
5505024d^{14}+66846720d^{13}+376373248d^{12}+1307049984d^{11}+3139540992d^{10}\\
&
+5534562304d^9+7393403136d^8+7593124096d^7+6001005236d^6\\
& 
+3610232652d^5+1615182134d^4+516014087d^3+109626523d^2+13506249d\\
&
+699840,\\
c_6(d)&=
2097152d^{15}+28573696d^{14}+181665792d^{13}+715063296d^{12}+1949155328d^{11}\\
& 
+3898431488d^{10}+5912027136d^9+6925392128d^8+6322048032d^7\\
& 
+4502239636d^6+2485004860d^5+1046787214d^4+327172991d^3\\
& 
+72131011d^2+10147017d+699840.
\end{align*}
Since $c_5(d), c_6(d)$ have exclusively positive coefficients, we conclude that $\frac{c_5(d)}{c_6(d)} > 0$, and so we conclude that $\K_d  \leq 4d + \frac{3}{2} + \frac{16}{d}$.
\end{proof}

%c1
%(229376*d^10+1966080*d^9+7294976*d^8+15323136*d^7+20124288*d^6+17555072*d^5+11102496*d^4+5917032*d^3+2803847*d^2+933639*d+139968)
%c2
%(-16777216*d^23-209715200*d^22-1265631232*d^21-4968939520*d^20-14345764864*d^19-32330973184*d^18-58457362432*d^17-85738966016*d^16-102477934592*d^15-100110810240*d^14-80143714368*d^13-52681506144*d^12-28456978128*d^11-12623832456*d^10-4590955884*d^9-1365827366*d^8-331953575*d^7-65949629*d^6-10735783*d^5-1431097*d^4-153709*d^3-12647*d^2-713*d-21)/
%c3
%524288*d^16+6029312*d^15+29491200*d^14+76218368*d^13+92368896*d^12-45438976*d^11-403496448*d^10-849085696*d^9-1125540408*d^8-1083739620*d^7-793966088*d^6-450174888*d^5-201247133*d^4-73160887*d^3-21407400*d^2-4340565*d-419904
%c4
%8388608*d^17+118489088*d^16+783810560*d^15+3223584768*d^14+9227272192*d^13+19497541632*d^12+31470530560*d^11+39594894336*d^10+39261180544*d^9+30802044560*d^8+19075944984*d^7+9245407152*d^6+3450501536*d^5+966637518*d^4+196684154*d^3+28701088*d^2+3266958*d+279936
%c5
%(5505024*d^14+66846720*d^13+376373248*d^12+1307049984*d^11+3139540992*d^10+5534562304*d^9+7393403136*d^8+7593124096*d^7+6001005236*d^6+3610232652*d^5+1615182134*d^4+516014087*d^3+109626523*d^2+13506249*d+699840)/%c6
%(2097152*d^15+28573696*d^14+181665792*d^13+715063296*d^12+1949155328*d^11+3898431488*d^10+5912027136*d^9+6925392128*d^8+6322048032*d^7+4502239636*d^6+2485004860*d^5+1046787214*d^4+327172991*d^3+72131011*d^2+10147017*d+699840)	

We plot in Figure~\ref{Fig:0401} the values of $\K_d - 4d - \frac{3}{2}$ for $2 \leq d \leq 30$, as well as the upper bound given in Proposition~\ref{Prop:0401} for comparison.

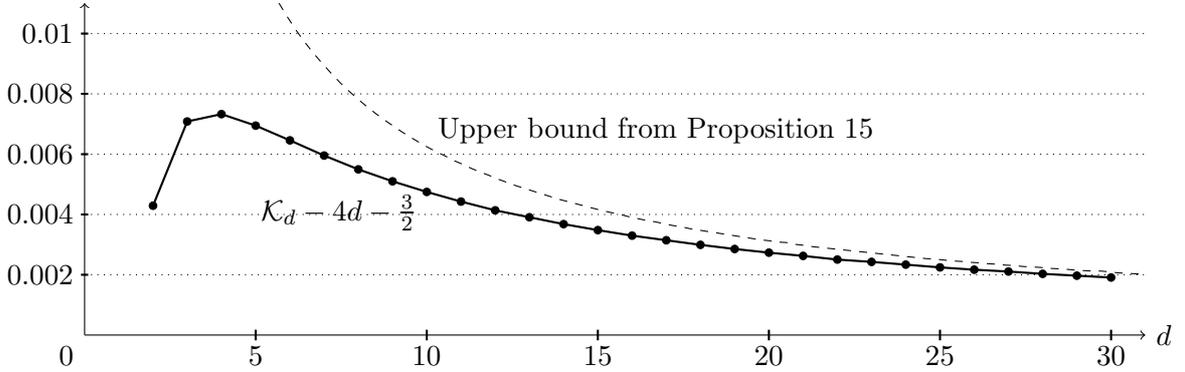
\begin{figure}[htb]
\begin{center}
\begin{tikzpicture}[scale=1, xscale = 0.45, yscale = 400, font = \small, word node/.style={font=\small}]]
\def\xlb{0};
\def\xub{30};
\def\ylb{0};
\def\yub{0.01};
\def\xbuf{1};
\def\ybuf{0.001};
\draw [->] (0,0) -- (\xub + \xbuf,0);
\draw [->] (0,0) -- (0, \yub + \ybuf);
\foreach \x in {\xlb ,5, ...,\xub}
{
\ifthenelse{\NOT 0 = \x}{\draw[thick](\x ,-0.005pt) -- (\x ,0.005pt);}{}
\ifthenelse{\NOT 0 = \x}{\node[anchor=north] at (\x,0) (label) {{$\x$}};}{}
}
\foreach \y in {0.002, 0.004, ..., \yub}
{
\draw[thick](-3pt, \y ) -- (3pt, \y);
\draw[dotted](0, \y ) -- (\xub + \xbuf, \y);
\node[anchor=east] at (0,\y) (label) {{ $\y$}};
}
\node[anchor=north east] at (0,0) (label3) { $0$};
\node[anchor=west] at (\xub + \xbuf,0) (label3) {$d$};
%\node[anchor=south] at (0, \yub+\ybuf) (label3) {$\footnotesize  \K_d - 4d - \frac{3}{2}$};

\draw[thick](2,0.004290)--(3,0.007080)--(4,0.007320)--(5,0.006950)--(6,0.006450)--(7,0.00
5950)--(8,0.005500)--(9,0.005100)--(10,0.004740)--(11,0.004420)--(12,0.004140
)--(13,0.003900)--(14,0.003670)--(15,0.003480)--(16,0.003300)--(17,0.003140)-
-(18,0.002990)--(19,0.002850)--(20,0.002730)--(21,0.002620)--(22,0.002510)--(
23,0.002420)--(24,0.002330)--(25,0.002250)--(26,0.002170)--(27,0.002100)--(28
,0.002030)--(29,0.001970)--(30,0.001910);
\def\y{0.5};
\foreach \position in {(2,0.004290),(3,0.007080),(4,0.007320),(5,0.006950),(6,0.006450),(7,0.005950)
,(8,0.005500),(9,0.005100),(10,0.004740),(11,0.004420),(12,0.004140),(13,0.00
3900),(14,0.003670),(15,0.003480),(16,0.003300),(17,0.003140),(18,0.002990),(
19,0.002850),(20,0.002730),(21,0.002620),(22,0.002510),(23,0.002420),(24,0.00
2330),(25,0.002250),(26,0.002170),(27,0.002100),(28,0.002030),(29,0.001970),(
30,0.001910)}
{\node[draw, circle, inner sep=0pt, minimum size = 0.1cm, fill] at \position {};}
\node[anchor = north east] at (10,0.005) {$\K_d - 4d - \frac{3}{2}$};

\draw[dashed, domain= {1/(16*(\yub+\ybuf))} : \xbuf+\xub, samples = 100] plot (\x, {1 / (16 * \x)} );
\node[anchor = south west] at (10,0.006) {Upper bound from Proposition~\ref{Prop:0401}};

\end{tikzpicture}
\end{center}
\caption{Plotting $\K_d-4d-\frac{3}{2}$ for $2 \leq d \leq 30$ and the upper bound from Proposition~\ref{Prop:0401}.}
\label{Fig:0401}
\end{figure}

\subsection{Relating trees and hypercube decompositions}

We end this section by describing a mapping from plane rooted trees to hypercube decompositions, which will also add some perspective to the bounds we found for $\K_d$ in Proposition~\ref{Prop:0401}.

Given an integer $d \geq 1$, let $\T_{d}$ be the set of plane rooted trees where each internal node has a label from $[d]$ and at least $2$ children, while the leaves of the tree are unlabelled. Furthermore, let $\T_{d,n} \subseteq \T_{d}$ denote the set of trees with exactly $n$ leaves. Then we can define a tree-to-decomposition mapping $\Psi : \T_{d,n} \to \S_{d,n}$ recursively as follows:

\begin{itemize}
\item
$\Psi$ maps the tree with a single leaf node to the trivial decomposition $\set{(0,1)^d}$.
\item
Now suppose $T \in \T_{d,n}$ has root node labelled $i$ with $r$ children. Let $T_1, \ldots, T_r$ be the subtrees of the root node ordered from left to right. Also, for each $j \in [r]$, let $B_j = \set{ x \in (0,1)^d : \frac{j-1}{r} < x_i < \frac{j}{r}}$. (Notice that $\set{B_1, B_2, \ldots, B_r} = H_{i,r}( (0,1)^d))$.) Define
\[
\Psi(T) = \bigcup_{j=1}^r \set{ \scale_{(0,1)^d \to B_j}(R) : R \in \Psi(T_j)}.
\]
\end{itemize}

Figure~\ref{Fig:0402} gives an example of this mapping.

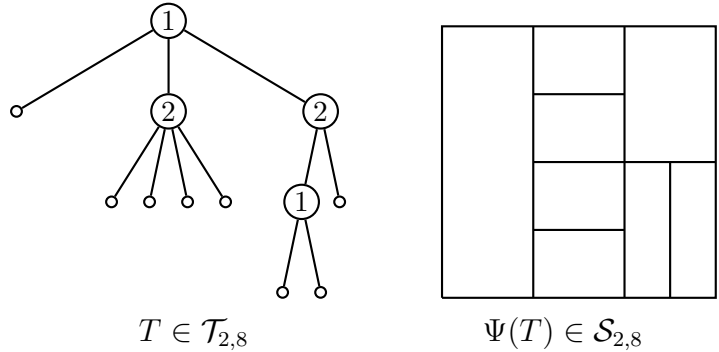
\begin{figure}[h]
\begin{center}
\begin{tabular}{cc}
\begin{tikzpicture}
[xscale=0.5, yscale = 0.6,thick,main node/.style={circle,inner sep=0.5mm,draw,font=\small\sffamily}]
  \node[main node] at (4,6) (0) {$1$};
    \node[main node] at (0,4) (1) {};
    \node[main node] at (4,4) (2) {$2$};
  \node[main node] at (8,4) (3){$2$};
      \node[main node] at (2.5,2) (21) {};
      \node[main node] at (3.5,2) (22) {};
      \node[main node] at (4.5,2) (23) {};
      \node[main node] at (5.5,2) (24) {};
  
      \node[main node] at (7.5,2) (31) {$1$};
            \node[main node] at (8.5,2) (32) {};
      \node[main node] at (7,0) (311) {};
      \node[main node] at (8,0) (312) {};
                     
  \path[every node/.style={font=\sffamily}]
(0) edge (1)
(0) edge (2)
(0) edge (3)
(2) edge (21)
(2) edge (22)
(2) edge (23)
(2) edge (24)
(3) edge (31)
(3) edge (32)
(31) edge (311)
(31) edge (312);
   ;
\end{tikzpicture}
~~
&
~~
\begin{tikzpicture}
[scale=0.3, thick]

\draw(0,0) -- (12,0) -- (12,12) -- (0,12) -- (0,0);
\draw(4,0) -- (4,12);
\draw(8,0) -- (8,12);
\draw(4,6) -- (8,6);
\draw(4,3) -- (8,3);
\draw(4,9) -- (8,9);
\draw(12,6) -- (8,6);
\draw(10,0) -- (10,6);
\end{tikzpicture}\\
$T \in \T_{2,8}$ & $\Psi(T) \in \S_{2,8}$ \\
\end{tabular}
\caption{Illustrating the mapping $\Psi$ from trees to hypercube decompositions.}\label{Fig:0402}
\end{center}
\end{figure}

The mapping $\Psi$ is a natural extension of another tree-to-decomposition mapping that Bagherzadeh,  Bremner, and the author used to study a generalization of Catalan numbers~\cite{AuBB20}. It is not hard to see that the mapping $\Psi$ is onto --- given $S \in \S_{d,n}$, we can use the sequence of splitting operations that resulted in $S$ to generate a tree $T \in \T_{d,n}$ where $\Psi(T) = S$. On the other hand, $\Psi$ is not one-to-one. Figure~\ref{Fig:0403} gives two types of situations where two distinct trees are mapped to the same decomposition.

\begin{figure}[h!]
\begin{center}
\begin{tabular}{ccc}
\begin{tikzpicture}
[scale=0.34,thick,main node/.style={circle,inner sep=0.5mm,draw,font=\small\sffamily}]
  \node[main node] at (4,6) (0) {$1$};
  \node[main node] at (1.5,4) (l1){};
    \node[main node] at (2.5,4) (l2){};
      \node[main node] at (3.5,4) (l3){};
        \node[main node] at (4.5,4) (l4){};
          \node[main node] at (5.5,4) (l5){};
            \node[main node] at (6.5,4) (l6){};
  \path[every node/.style={font=\sffamily}]
    (0) edge (l1)
        (0) edge (l2)
            (0) edge (l3)
                (0) edge (l5)
                    (0) edge (l4)
                        (0) edge (l6)
    ;
\end{tikzpicture}

~~~&~~~
\begin{tikzpicture}
[scale=0.34,thick,main node/.style={circle,inner sep=0.5mm,draw,font=\small\sffamily}]
  \node[main node] at (4,6) (0) {$1$};
    \node[main node] at (2,4) (l) {$1$};
    \node[main node] at (6,4) (r) {$1$};
  \node[main node] at (1,2) (l1){};
    \node[main node] at (2,2) (l2){};
      \node[main node] at (3,2) (l3){};
        \node[main node] at (5,2) (l4){};
          \node[main node] at (6,2) (l5){};
            \node[main node] at (7,2) (l6){};
  \path[every node/.style={font=\sffamily}]
    (0) edge (l)
    (0) edge (r)
                (l) edge (l1)
                 (l) edge (l2)
            (l) edge (l3)
                (r) edge (l5)
                    (r) edge (l4)
                        (r) edge (l6)
    ;
\end{tikzpicture}

~~~&~~~

\begin{tikzpicture}
[scale=0.3,thick]
\def\x{0.5}
\draw (0,0) -- (6,0);
\draw (0,{\x}) -- (0, {-\x});
\draw (1,{\x}) -- (1, {-\x});
\draw (2,{\x}) -- (2, {-\x});
\draw (3,{\x}) -- (3, {-\x});
\draw (4,{\x}) -- (4, {-\x});
\draw (5,{\x}) -- (5, {-\x});
\draw (6,{\x}) -- (6, {-\x});
\end{tikzpicture}

\\
$T_1 \in \T_{1,6}$ & 
$T_2 \in \T_{1,6}$ & 
$\Psi(T_1) = \Psi(T_2) \in \S_{1,6}$  \\
\\
\hline
\\

\begin{tikzpicture}
[scale=0.34,thick,main node/.style={circle,inner sep=0.5mm,draw,font=\small\sffamily}]
  \node[main node] at (4,6) (0) {$1$};
    \node[main node] at (2,4) (l) {$2$};
    \node[main node] at (6,4) (r) {$2$};
  \node[main node] at (1,2) (l1){};
    \node[main node] at (2,2) (l2){};
      \node[main node] at (3,2) (l3){};
        \node[main node] at (5,2) (l4){};
          \node[main node] at (6,2) (l5){};
            \node[main node] at (7,2) (l6){};
  \path[every node/.style={font=\sffamily}]
    (0) edge (l)
    (0) edge (r)
                (l) edge (l1)
                 (l) edge (l2)
            (l) edge (l3)
                (r) edge (l5)
                    (r) edge (l4)
                        (r) edge (l6)
    ;
\end{tikzpicture}

~~~&~~~
\begin{tikzpicture}
[scale=0.34,thick,main node/.style={circle,inner sep=0.5mm,draw,font=\small\sffamily}]
  \node[main node] at (4,6) (0) {$2$};
    \node[main node] at (1.5,4) (l) {$1$};
        \node[main node] at (4,4) (c) {$1$};
    \node[main node] at (6.5,4) (r) {$1$};
  \node[main node] at (1,2) (l1){};
    \node[main node] at (2,2) (l2){};
      \node[main node] at (3.5,2) (l3){};
        \node[main node] at (4.5,2) (l4){};
          \node[main node] at (6,2) (l5){};
            \node[main node] at (7,2) (l6){};
  \path[every node/.style={font=\sffamily}]
    (0) edge (l)
        (0) edge (c)
    (0) edge (r)
                (l) edge (l1)
                 (l) edge (l2)
            (c) edge (l3)
                (c) edge (l4)
                    (r) edge (l5)
                        (r) edge (l6)
    ;
    
\end{tikzpicture}

~~~&~~~
\begin{tikzpicture}
[scale=0.15, thick]
\draw(0,0) -- (12,0) -- (12,12) -- (0,12) -- (0,0);
\draw(0,4) -- (12,4);
\draw(0,8) -- (12,8);
\draw(6,0) -- (6,12);
\end{tikzpicture}
\\
$T_3 \in \T_{2,6}$ & 
$T_4 \in \T_{2,6}$ & 
$\Psi(T_3) = \Psi(T_4) \in \S_{2,6}$  \\
\end{tabular}
\caption{Two types of situations where $\Psi : \T_{d,n} \to \S_{d,n}$ is not one-to-one.}\label{Fig:0403}
\end{center}
\end{figure}
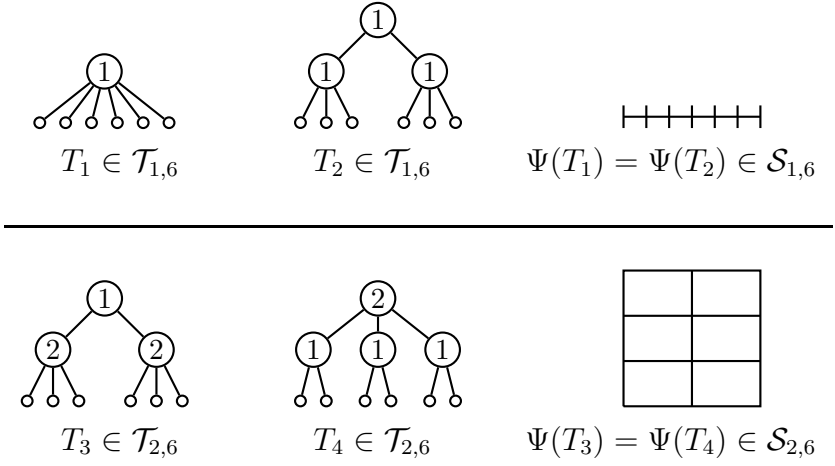

In general, if we let $t_d(n) = |\T_{d,n}|$, then we see that $t_d(n) \geq s_d(n)$ for every $d,n \geq 1$, with the inequality being strict for all $n \geq 4$. We remark that when $d=1$, we can consider the nodes of the trees in $\T_{1,n}$ as being unlabelled, and so $t_1(n)$ gives the well-studied small Schr\"oder numbers~\cite[\href{https://oeis.org/A001003}{A001003}]{OEIS}. Thus, the sequences $t_d(n)$ can be seen as a generalization of small Schr\"oder numbers, which the author recently studied in another manuscript~\cite{Au21}. In particular, we have the following for the growth rate of $t_d(n)$:

\begin{proposition}[\cite{Au21}, Proposition 4]
For every $d \geq 1$,
\[
\lim_{n \to \infty} \frac{t_d(n+1)}{t_d(n)} = 2d+1+2\sqrt{d^2+d}.
\]
\end{proposition}

Thus, while $\displaystyle\lim_{d \to \infty} \lim_{n \to \infty} \frac{s_d(n+1)}{s_d(n)} = 4d+\frac{3}{2}$, $\displaystyle\lim_{d \to \infty} \lim_{n \to \infty} \frac{t_d(n+1)}{t_d(n)} = 4d+2$. Also, since the growth rate of $t_d(n)$ is strictly greater than that of $s_d(n)$, it follows that $\displaystyle\lim_{n \to \infty} \frac{s_d(n)}{t_d(n)} = 0$ for every $d \geq 1$.

%Computation for other parameters?

\section{Relating Decompositions and NECS}\label{Sec:05}

An immediate consequence of Theorem~\ref{Thm:0101} is that $s_1(n) = c(n)$ for all $n \geq 1$. In this section, we highlight some combinatorial connections between hypercube decompositions and NECS. 

\subsection{A bijection between $\S_{1,n}$ and $\C_n$}

Given an NECS $C = \set{ \res{a_i}{n_i} : i \in [k]}$, define $\gcd(C)$ to be the greatest common divisor of the moduli $n_1, \ldots, n_k$. Furthermore, when we discuss the gcd of a decomposition in $S \in \S_{1,n}$, we'll slightly abuse notation and write $\gcd(S) = r$ (instead of $\gcd(S) = (r)$, for ease of comparing gcd's of decompositions and NECS.

Next, we define the mapping $\Phi : \S_{1,n} \to \C_n$ recursively as follows:

\begin{itemize}
\item
($n=1$) $\Phi$ maps the trivial decomposition $\set{(0,1)}$ to $\set{ \res{0}{1}}$, the NECS with a single residual class.
\item
($n \geq 2$) Given $S \in \S_{1,n}$, let $r = \gcd(S)$, and $B_j = (\frac{j}{r}, \frac{j+1}{r})$ for every $j \in \set{0, \ldots, r-1}$. Given that $\gcd(S) = r$, we know that for every region $R \in S$, there is a unique $j$ where $R \subseteq B_j$. Thus, 
\[
S_j = \set{ \scale_{B_j \to (0,1)}(R) : R \in S, R \subseteq B_j},
\]
is a decomposition in its own right. We then define
\[
\Phi(S) = \bigcup_{j=0}^{r-1} \set{ E_{j,r}(T) : T \in \Phi(S_j)}.
\]
(Recall that $E_{j,r}( \res{a}{n} ) = \res{jn+a}{rn}$, as defined in Section~\ref{Sec:0103}.)
\end{itemize}

\begin{figure}[pt!]
\begin{center}
\begin{tabular}{l|l|c|c}
$S$ & $\Phi(S)$ & gcd & lcm \\
\hline
%n=1
\begin{tikzpicture}
[scale=5,thick, font=\small\sffamily]
\def\x{0.02}
\draw (0,0) -- (1,0);
\draw (0,{\x}) -- (0, {-\x});
\node[anchor = south] at ({0},{\x}) (1) {$0$};
\draw (1,{\x}) -- (1, {-\x});
\node[anchor = south] at ({1},{\x}) (1) {$1$};
\end{tikzpicture}
& $\set{ \res{0}{1}}$ & $1$ & $1$  \\
%n=2
\begin{tikzpicture}
[scale=5,thick, font=\small\sffamily]
\def\x{0.02}
\draw (0,0) -- (1,0);
\draw (0,{\x}) -- (0, {-\x});
\node[anchor = south] at ({0},{\x}) (1) {$0$};
\draw (1,{\x}) -- (1, {-\x});
\node[anchor = south] at ({1},{\x}) (1) {$1$};
\draw ({1/2},{\x}) -- ({1/2}, {-\x});
\node[anchor = south] at ({1/2},{\x}) (1) {$\frac{1}{2}$};
\end{tikzpicture}
& $\set{ \res{0}{2}, \res{1}{2}}$ & $2$ & $2$ \\
%n=3
\begin{tikzpicture}
[scale=5,thick, font=\small\sffamily]
\def\x{0.02}
\draw (0,0) -- (1,0);
\draw (0,{\x}) -- (0, {-\x});
\node[anchor = south] at ({0},{\x}) (1) {$0$};
\draw (1,{\x}) -- (1, {-\x});
\node[anchor = south] at ({1},{\x}) (1) {$1$};
\draw ({1/2},{\x}) -- ({1/2}, {-\x});
\node[anchor = south] at ({1/2},{\x}) (1) {$\frac{1}{2}$};
\draw ({1/4},{\x}) -- ({1/4}, {-\x});
\node[anchor = south] at ({1/4},{\x}) (1) {$\frac{1}{4}$};
\end{tikzpicture}
& $\set{ \res{0}{4}, \res{2}{4}, \res{1}{2}}$ & $2$ & $4$  \\
\begin{tikzpicture}
[scale=5,thick, font=\small\sffamily]
\def\x{0.02}
\draw (0,0) -- (1,0);
\draw (0,{\x}) -- (0, {-\x});
\node[anchor = south] at ({0},{\x}) (1) {$0$};
\draw (1,{\x}) -- (1, {-\x});
\node[anchor = south] at ({1},{\x}) (1) {$1$};
\draw ({1/2},{\x}) -- ({1/2}, {-\x});
\node[anchor = south] at ({1/2},{\x}) (1) {$\frac{1}{2}$};
\draw ({3/4},{\x}) -- ({3/4}, {-\x});
\node[anchor = south] at ({3/4},{\x}) (1) {$\frac{3}{4}$};
\end{tikzpicture}
& $\set{ \res{0}{2}, \res{1}{4}, \res{3}{4}}$ & $2$ & $4$  \\
\begin{tikzpicture}
[scale=5,thick, font=\small\sffamily]
\def\x{0.02}
\draw (0,0) -- (1,0);
\draw (0,{\x}) -- (0, {-\x});
\node[anchor = south] at ({0},{\x}) (1) {$0$};
\draw (1,{\x}) -- (1, {-\x});
\node[anchor = south] at ({1},{\x}) (1) {$1$};
\draw ({1/3},{\x}) -- ({1/3}, {-\x});
\node[anchor = south] at ({1/3},{\x}) (1) {$\frac{1}{3}$};
\draw ({2/3},{\x}) -- ({2/3}, {-\x});
\node[anchor = south] at ({2/3},{\x}) (1) {$\frac{2}{3}$};
\end{tikzpicture}
& $\set{ \res{0}{3}, \res{1}{3}, \res{2}{3}}$ & $3$ & $3$  \\
%n=4
\begin{tikzpicture}
[scale=5,thick, font=\small\sffamily]
\def\x{0.02}
\draw (0,0) -- (1,0);
\draw (0,{\x}) -- (0, {-\x});
\node[anchor = south] at ({0},{\x}) (1) {$0$};
\draw (1,{\x}) -- (1, {-\x});
\node[anchor = south] at ({1},{\x}) (1) {$1$};
\draw ({1/2},{\x}) -- ({1/2}, {-\x});
\node[anchor = south] at ({1/2},{\x}) (1) {$\frac{1}{2}$};
\draw ({1/6},{\x}) -- ({1/6}, {-\x});
\node[anchor = south] at ({1/6},{\x}) (1) {$\frac{1}{6}$};
\draw ({1/3},{\x}) -- ({1/3}, {-\x});
\node[anchor = south] at ({1/3},{\x}) (1) {$\frac{2}{6}$};
\end{tikzpicture}
& $\set{ \res{0}{6}, \res{2}{6}, \res{4}{6}, \res{1}{2}}$ & $2$ & $6$  \\

\begin{tikzpicture}
[scale=5,thick, font=\small\sffamily]
\def\x{0.02}
\draw (0,0) -- (1,0);
\draw (0,{\x}) -- (0, {-\x});
\node[anchor = south] at ({0},{\x}) (1) {$0$};
\draw (1,{\x}) -- (1, {-\x});
\node[anchor = south] at ({1},{\x}) (1) {$1$};
\draw ({1/2},{\x}) -- ({1/2}, {-\x});
\node[anchor = south] at ({1/2},{\x}) (1) {$\frac{1}{2}$};
\draw ({2/3},{\x}) -- ({2/3}, {-\x});
\node[anchor = south] at ({2/3},{\x}) (1) {$\frac{4}{6}$};
\draw ({5/6},{\x}) -- ({5/6}, {-\x});
\node[anchor = south] at ({5/6},{\x}) (1) {$\frac{5}{6}$};
\end{tikzpicture}
& $\set{ \res{0}{2}, \res{1}{6}, \res{3}{6}, \res{5}{6}}$ & $2$ & $6$  \\
\begin{tikzpicture}
[scale=5,thick, font=\small\sffamily]
\def\x{0.02}
\draw (0,0) -- (1,0);
\draw (0,{\x}) -- (0, {-\x});
\node[anchor = south] at ({0},{\x}) (1) {$0$};
\draw (1,{\x}) -- (1, {-\x});
\node[anchor = south] at ({1},{\x}) (1) {$1$};
\draw ({1/2},{\x}) -- ({1/2}, {-\x});
\node[anchor = south] at ({1/2},{\x}) (1) {$\frac{1}{2}$};
\draw ({1/4},{\x}) -- ({1/4}, {-\x});
\node[anchor = south] at ({1/4},{\x}) (1) {$\frac{1}{4}$};
\draw ({1/8},{\x}) -- ({1/8}, {-\x});
\node[anchor = south] at ({1/8},{\x}) (1) {$\frac{1}{8}$};
\end{tikzpicture}
& $\set{ \res{0}{8}, \res{4}{8}, \res{2}{4}, \res{1}{2}}$ & $2$ & $8$  \\
\begin{tikzpicture}
[scale=5,thick, font=\small\sffamily]
\def\x{0.02}
\draw (0,0) -- (1,0);
\draw (0,{\x}) -- (0, {-\x});
\node[anchor = south] at ({0},{\x}) (1) {$0$};
\draw (1,{\x}) -- (1, {-\x});
\node[anchor = south] at ({1},{\x}) (1) {$1$};
\draw ({1/2},{\x}) -- ({1/2}, {-\x});
\node[anchor = south] at ({1/2},{\x}) (1) {$\frac{1}{2}$};
\draw ({1/4},{\x}) -- ({1/4}, {-\x});
\node[anchor = south] at ({1/4},{\x}) (1) {$\frac{1}{4}$};
\draw ({3/8},{\x}) -- ({3/8}, {-\x});
\node[anchor = south] at ({3/8},{\x}) (1) {$\frac{3}{8}$};
\end{tikzpicture}
& $\set{ \res{0}{4}, \res{2}{8}, \res{6}{8}, \res{1}{2}}$ & $2$ & $8$  \\
\begin{tikzpicture}
[scale=5,thick, font=\small\sffamily]
\def\x{0.02}
\draw (0,0) -- (1,0);
\draw (0,{\x}) -- (0, {-\x});
\node[anchor = south] at ({0},{\x}) (1) {$0$};
\draw (1,{\x}) -- (1, {-\x});
\node[anchor = south] at ({1},{\x}) (1) {$1$};
\draw ({1/2},{\x}) -- ({1/2}, {-\x});
\node[anchor = south] at ({1/2},{\x}) (1) {$\frac{1}{2}$};
\draw ({3/4},{\x}) -- ({3/4}, {-\x});
\node[anchor = south] at ({3/4},{\x}) (1) {$\frac{3}{4}$};
\draw ({5/8},{\x}) -- ({5/8}, {-\x});
\node[anchor = south] at ({5/8},{\x}) (1) {$\frac{5}{8}$};
\end{tikzpicture}
& $\set{ \res{0}{2}, \res{1}{8}, \res{5}{8}, \res{3}{4}}$ & $2$ & $8$  \\
\begin{tikzpicture}
[scale=5,thick, font=\small\sffamily]
\def\x{0.02}
\draw (0,0) -- (1,0);
\draw (0,{\x}) -- (0, {-\x});
\node[anchor = south] at ({0},{\x}) (1) {$0$};
\draw (1,{\x}) -- (1, {-\x});
\node[anchor = south] at ({1},{\x}) (1) {$1$};
\draw ({1/2},{\x}) -- ({1/2}, {-\x});
\node[anchor = south] at ({1/2},{\x}) (1) {$\frac{1}{2}$};
\draw ({3/4},{\x}) -- ({3/4}, {-\x});
\node[anchor = south] at ({3/4},{\x}) (1) {$\frac{3}{4}$};
\draw ({7/8},{\x}) -- ({7/8}, {-\x});
\node[anchor = south] at ({7/8},{\x}) (1) {$\frac{7}{8}$};
\end{tikzpicture}
& $\set{ \res{0}{2}, \res{1}{4}, \res{3}{8}, \res{7}{8}}$ & $2$ & $8$  \\
\begin{tikzpicture}
[scale=5,thick, font=\small\sffamily]
\def\x{0.02}
\draw (0,0) -- (1,0);
\draw (0,{\x}) -- (0, {-\x});
\node[anchor = south] at ({0},{\x}) (1) {$0$};
\draw (1,{\x}) -- (1, {-\x});
\node[anchor = south] at ({1},{\x}) (1) {$1$};
\draw ({1/3},{\x}) -- ({1/3}, {-\x});
\node[anchor = south] at ({1/3},{\x}) (1) {$\frac{1}{3}$};
\draw ({2/3},{\x}) -- ({2/3}, {-\x});
\node[anchor = south] at ({2/3},{\x}) (1) {$\frac{2}{3}$};
\draw ({1/6},{\x}) -- ({1/6}, {-\x});
\node[anchor = south] at ({1/6},{\x}) (1) {$\frac{1}{6}$};
\end{tikzpicture}
& $\set{ \res{0}{6}, \res{3}{6}, \res{1}{3}, \res{2}{3}}$ & $3$ & $6$  \\
\begin{tikzpicture}
[scale=5,thick, font=\small\sffamily]
\def\x{0.02}
\draw (0,0) -- (1,0);
\draw (0,{\x}) -- (0, {-\x});
\node[anchor = south] at ({0},{\x}) (1) {$0$};
\draw (1,{\x}) -- (1, {-\x});
\node[anchor = south] at ({1},{\x}) (1) {$1$};
\draw ({1/3},{\x}) -- ({1/3}, {-\x});
\node[anchor = south] at ({1/3},{\x}) (1) {$\frac{1}{3}$};
\draw ({2/3},{\x}) -- ({2/3}, {-\x});
\node[anchor = south] at ({2/3},{\x}) (1) {$\frac{2}{3}$};
\draw ({3/6},{\x}) -- ({3/6}, {-\x});
\node[anchor = south] at ({3/6},{\x}) (1) {$\frac{3}{6}$};
\end{tikzpicture}
& $\set{ \res{0}{3}, \res{1}{6}, \res{4}{6}, \res{2}{3}}$ & $3$ & $6$  \\
\begin{tikzpicture}
[scale=5,thick, font=\small\sffamily]
\def\x{0.02}
\draw (0,0) -- (1,0);
\draw (0,{\x}) -- (0, {-\x});
\node[anchor = south] at ({0},{\x}) (1) {$0$};
\draw (1,{\x}) -- (1, {-\x});
\node[anchor = south] at ({1},{\x}) (1) {$1$};
\draw ({1/3},{\x}) -- ({1/3}, {-\x});
\node[anchor = south] at ({1/3},{\x}) (1) {$\frac{1}{3}$};
\draw ({2/3},{\x}) -- ({2/3}, {-\x});
\node[anchor = south] at ({2/3},{\x}) (1) {$\frac{2}{3}$};
\draw ({5/6},{\x}) -- ({5/6}, {-\x});
\node[anchor = south] at ({5/6},{\x}) (1) {$\frac{5}{6}$};
\end{tikzpicture}
& $\set{ \res{0}{3}, \res{1}{3}, \res{2}{6}, \res{5}{6}}$ & $3$ & $6$  \\
\begin{tikzpicture}
[scale=5,thick, font=\small\sffamily]
\def\x{0.02}
\draw (0,0) -- (1,0);
\draw (0,{\x}) -- (0, {-\x});
\node[anchor = south] at ({0},{\x}) (1) {$0$};
\draw (1,{\x}) -- (1, {-\x});
\node[anchor = south] at ({1},{\x}) (1) {$1$};
\draw ({1/4},{\x}) -- ({1/4}, {-\x});
\node[anchor = south] at ({1/4},{\x}) (1) {$\frac{1}{4}$};
\draw ({2/4},{\x}) -- ({2/4}, {-\x});
\node[anchor = south] at ({2/4},{\x}) (1) {$\frac{2}{4}$};
\draw ({3/4},{\x}) -- ({3/4}, {-\x});
\node[anchor = south] at ({3/4},{\x}) (1) {$\frac{3}{4}$};
\end{tikzpicture}
& $\set{ \res{0}{4}, \res{1}{4}, \res{2}{4}, \res{3}{4}}$ & $4$ & $4$ 
\end{tabular}
\caption{Illustrating the mapping $\Phi : \S_{1,n} \to \C_n$ for $n \leq 4$.}\label{Fig:0501}
\end{center}
\end{figure}

Figure~\ref{Fig:0501} illustrates the mapping $\Phi$ applied to elements in $\S_{1,n}$ for $n \leq 4$. Next, we prove a result that will help us show that $\Phi$ is in fact a bijection.

\begin{lemma}\label{Lem:0501}
Let $S \in \S_{1,n}$. If $\gcd(S) = r$, then $\gcd(\Phi(S)) = r$.
\end{lemma}

\begin{proof}
We prove our claim by induction on $n$. When $n=1$, then $S$ must be the trivial decomposition $\set{ (0,1)}$, and the claim holds as $\gcd(S) = \gcd(\Phi(S))= 1$. 

Now assume that $n \geq 2$, and so $r \geq 2$. Consider the decompositions $S_0, \ldots, S_{r-1}$ as defined in the definition of $\Phi$. Since $r = \gcd(S)$ is chosen maximally, it follows that
\[
\gcd \set{ \gcd(S_j) : j \in \set{0,\ldots, r-1}} = 1.
\]
(The outer $\gcd$ is the ordinary greatest common divisor operation applied to a set of natural numbers.) By the inductive hypothesis, $\gcd(S_j) = \gcd(\Phi(S_j))$ for all $j$, and so we obtain that
\[
\gcd \set{ \gcd(\Phi(S_j)) : j \in \set{0,\ldots, r-1}} = 1.
\]
Since $E_{j,r}$ multiplies each modulus of the residual classes in $\Phi(S_j)$ by $r$, we see that 
\[
\gcd\set{ E_{j,r}(T) : T \in \Phi(S_j)} = r \gcd(\Phi(S_j))
\]
for every $j \in \set{0, \ldots, r-1}$, and so it follows that 
\[
\gcd(\Phi(S)) = \gcd\set{ E_{j,r}(T) : j \in \set{0,\ldots, r-1}, T \in \Phi(S_j)} = r. \mbox{\qedhere}
\]
\end{proof}

Since $\Phi$ is gcd-preserving from Lemma~\ref{Lem:0501}, we are able to use a single column in Figure~\ref{Fig:0501} to indicate the gcd of both the input decomposition and the output NECS. Next, we prove that $\Phi$ is indeed a bijection.

\begin{proposition}
$\Phi : \S_{1,n} \to \C_n$ is a bijection.
\end{proposition}

\begin{proof}
From Theorem~\ref{Thm:0101}, we know that $|\S_{1,n}| = |\C_n|$, and so every function from $\S_{1,n}$ to $\C_n$ is either both one-to-one and onto, or neither. Thus, it suffices to show that $\Phi$ is one-to-one, which we will prove by induction on $n$. The base case $n=1$ obviously holds. 

For a contradiction, suppose there exist distinct $S, S' \in \S_{1,n}$ where $\Phi(S) = \Phi(S')$. From Lemma~\ref{Lem:0501}, $\Phi(S) = \Phi(S')$ implies that $\gcd(S) = \gcd(S') = r$ for some integer $r \geq 2$. Consider the decompositions $S_0, \ldots, S_r$ and $S'_0, \ldots, S'_r$ as defined in the definition of $\Phi$. Now $S \neq S'$ implies that there exists an index $j$ where $S_j \neq S'_j$. On the other hand, $\Phi(S) = \Phi(S')$ implies that $\Phi(S_j) = \Phi(S'_j)$. Since $S_j$ consists of strictly fewer regions than $S$, this violates our inductive hypothesis. Thus, it follows that $\Phi$ is bijective.
\end{proof}

\subsection{Counting decompositions/NECS with a given LCM}

We have established a map $\Phi$ between $1$-dimensional decompositions and NECS that is not only bijective, but also preserves some key combinatorial properties, such as the gcd as shown in Lemma~\ref{Lem:0501}. Thus, one can see hypercube decompositions as a generalization of NECS, and any result we prove for hypercube decompositions may also specialize to a corresponding implication for NECS. In this section, we provide one such example.

Similar to how we defined the gcd of a given hypercube decomposition, we can also define its lcm. Given a  $S \in \S_{d}$, we say that $\lcm(S) = (r_1, \ldots, r_d)$ if
\begin{itemize}
\item
$S \preceq D_{(r_1, \ldots, r_d)}$;
\item
there doesn't exist $(r_1',\ldots, r_d') \neq (r_1, \ldots, r_d)$ where $(r_1',\ldots, r_d') \leq (r_1, \ldots, r_d)$ and $S \preceq D_{(r_1', \ldots, r_d')}$.
\end{itemize}

Using the decompositions in Figure~\ref{Fig:0202} as examples, we have $\lcm(S_1) = (3,2)$, $\lcm(S_2) = (6,4)$, and $\lcm(S_3) = (6, 12)$. Then, given $r_1, \ldots, r_d \in \mN$, define the sets
\begin{align*}
\G_{(r_1, \ldots, r_d)}&=  \set{ S \in \S_d : S \preceq D_{(r_1, \ldots, r_d)} },\\
\Hc_{(r_1, \ldots, r_d)} &=  \set{ S \in \S_d : \lcm(S) =(r_1, \ldots, r_d) }.
\end{align*}

We also define $g(r_1, \ldots, r_d) = |\G_{(r_1, \ldots, r_d)}|$ and $h(r_1, \ldots, r_d) = |\Hc_{(r_1, \ldots, r_d)}|$. Since $\lcm(S) =(r_1, \ldots, r_d)$ implies that $S \preceq D_{(r_1, \ldots, r_d)}$, we see that $h(r_1, \ldots, r_d) \leq g(r_1,\ldots, r_d)$ for all $r_1, \ldots, r_d$. Also, when $r_1 = \cdots = r_d = 1$, it's obvious that $g(r_1,\ldots, r_d) = h(r_1,\ldots, r_d) = 1$. For the general case, we have the following recursive formulas:

\begin{proposition}\label{Prop:0501}
Given $r_1, \ldots, r_d \in \mN$,
\begin{itemize}
\item[(i)]
\[
g(r_1, \ldots, r_d) = 1 - \sum \left( \prod_{i=1}^d \mu(q_i) \right) g\left(\frac{r_1}{q_1}, \ldots, \frac{r_d}{q_d} \right)^{\prod_{i=1}^d q_i},
\]
where the sum is over $q_1, \ldots, q_d \in \mN$ where $q_i | r_i$ for every $i \in [d]$, and $\prod_{i=1}^d q_i \neq 1$. 
\item[(ii)]
\[
h(r_1, \ldots, r_d) = \sum \left( \prod_{i=1}^d \mu(q_i) \right) g\left(\frac{r_1}{q_1}, \ldots, \frac{r_d}{q_d} \right),
\]
where the sum is over $q_1, \ldots, q_d \in \mN$ where $q_i | r_i$ for every $i \in [d]$. 
\end{itemize}
\end{proposition}

\begin{proof}
We first prove $(i)$. Given $n \in \mN$, let $\mP(n)$ be the set of prime divisors of $n$ (e.g., $\mP(40) = \set{2,5}$). Now suppose $S \preceq D_{(r_1, \ldots, r_d)}$. Then either $S = \set{ (0,1)^d}$, or $S$  refines $\set{H_{i,p}((0,1)^d)}$ for some $i \in [d]$ and $p \in \mP(r_i)$. Thus, if we define
\[
\W_{i,p} = \set{ D \in \S_d : \set{H_{i,p}((0,1)^d)} \preceq D \preceq D_{(r_1, \ldots, r_d)} },
\]
then we see that
\[
\G_{(r_1,\ldots, r_d)}= \set{ \set{(0,1)^d}} \cup \bigcup_{i \in [d], p \in \mP(r_i)} \W_{i,p}.
\]
We apply the principle of inclusion-exclusion to the above to compute the size of $\G_{(r_1,\ldots, r_d)}$. For convenience, given a finite set $S \subseteq \mN$, we let $\pi(S)$ denote the product of all elements in $S$. Then notice that each intersection of a collection of $k$ sets from $\set{ \W_{i,p} : i \in [d], p \in \mP(r_i)}$ can be uniquely written as
\begin{equation}\label{Eq:0501}
\bigcap_{i=1}^d \left( \bigcap_{p \in P_i} \W_{i,p} \right) =  \set{ D \in \S_d : D_{(\pi(P_1), \ldots, \pi(P_d))} \preceq D \preceq D_{(r_1, \ldots, r_d)} },
\end{equation}
for some choice of $P_1, \ldots, P_d$ where $P_i \subseteq \mP(r_i)$ for every $i \in [d]$, and $\sum_{i=1}^d |P_i| = k$ (note that some of the $P_i$'s could be empty, in which case $\pi(P_i) = 1$). Now, the decomposition $D_{(\pi(P_1), \ldots, \pi(P_d))}$ consists of $m = \prod_{i=1}^d \pi(P_i)$ regions --- let's denote them by $B_1, \ldots, B_m$. Now, if $S$ is a decomposition that belongs to~\eqref{Eq:0501}, then $S \succeq D_{(\pi(P_1), \ldots, \pi(P_d))}$, and so
\[
S_j = \set{ \scale_{B_j \to (0,1)^d}(R) : R \in S, R \subseteq B_j}
\]
is a decomposition in its own right for every $j \in [m]$. Furthermore, $S \preceq D_{(r_1,\ldots, r_d)}$ implies that $S_j \preceq D_{(r_1/\pi(P_1), \ldots, r_d/\pi(P_d))}$. Thus, we see that there is a bijection between~\eqref{Eq:0501} and $\G_{(r_1/\pi(P_1), \ldots, r_d/\pi(P_d))}^m$, and so~\eqref{Eq:0501} has size $g\left( \frac{r_1}{\pi(P_1)}, \ldots, \frac{r_d}{\pi(P_d)}\right)^{\prod_{i=1}^d  \pi(P_i)}$. This gives
 \begin{align*}
& g(r_1,\ldots, r_d) \\
={}&1 +  \left| \bigcup_{i \in [d], p \in \mP(r_i)} \W_{i,p} \right|\\
={}& 1+ \sum_{k \geq 1} (-1)^{k-1}  \left( \sum_{\substack{P_1 \subseteq \mP(r_1), \ldots, P_d \subseteq \mP(r_i) \\ \sum_{i=1}^d |P_i|= k}} \left| \set{ D \in \S_d : D_{(\pi(P_1), \ldots, \pi(P_d))} \preceq D \preceq D_{(r_1, \ldots, r_d)} } \right|  \right) \\
={}& 1- \sum_{k \geq 1} (-1)^{k}  \left( \sum_{\substack{P_1 \subseteq \mP(r_1), \ldots, P_d \subseteq \mP(r_i) \\ \sum_{i=1}^d |P_i|= k}} g\left( \frac{r_1}{\pi(P_1)}, \ldots, \frac{r_d}{\pi(P_d)}\right)^{\prod_{i=1}^d \pi(P_i)} \right).
  \end{align*}
Next, notice that if we let $q_i = \pi(P_i)$ for every $i \in [d]$, then
\begin{equation}\label{Eq:0502}
  \prod_{i=1}^d \mu(q_i) =   \prod_{i=1}^d (-1)^{|P_i|} = (-1)^k.
\end{equation}
Thus, the sum above can be rewritten as
\[
1 - \sum \left( \prod_{i=1}^d \mu(q_i) \right) g\left(\frac{r_1}{q_1}, \ldots, \frac{r_d}{q_d} \right)^{\prod_{i=1}^d q_i},
\]
where the sum is over $q_1, \ldots, q_d$ where $q_i$ is a product of prime divisors of $r_i$ for every $i$, and the condition $\sum_{i=1}^d |P_i| = k \geq 1$ translates to $\prod_{i=1}^d q_i \neq 1$. Moreover, since $\mu(q_i) = 0$ when $q_i$ is divisible by a non-trivial square, the sum above would remain the same if we expanded it to include all $q_i$'s that are divisors of $r_i$, which results in the claimed formula.

Next, we prove $(ii)$ using a similar inclusion-exclusion argument as above. Given $S \in \G_{(r_1, \ldots, r_d)}$, then either $\lcm(S) = (r_1, \ldots, r_d)$ and $S \in \Hc_{(r_1,\ldots,r_d)}$, or there exists $i \in [d]$ and $p \in \mP(r_i)$ such that $S \in \G_{(r_1, \ldots, r_i/p, \ldots, r_d)}$. Thus,
\[
\Hc_{(r_1, \ldots, r_d)} = \G_{(r_1, \ldots, r_d)} \setminus  \bigcup_{i \in [d], p \in \mP(r_i)} \G_{(r_1, \ldots, r_i/p, \ldots, r_d)}.
\]
Now each intersection of a collection of $k$ sets from $\set{ \G_{(r_1, \ldots, r_i/p, \ldots, r_d)} : i \in [d], p \in \mP(r_i)}$ can be uniquely written as
\[
\bigcap_{i=1}^d \left( \bigcap_{p \in P_i} \G_{(r_1, \ldots, r_i/p, \ldots, r_d)} \right) =  \G_{(r_1/\pi(P_1), \ldots, (r_d/\pi(P_d))},
\]
for some choice of $P_1, \ldots, P_d$ where $P_i \subseteq \mP(r_i)$ for every $i \in [d]$, and $\sum_{i=1}^d |P_i| = k$. Thus, we obtain that
 \begin{align*}
& h(r_1,\ldots, r_d) \\
={}& g(r_1, \ldots, r_d) -  \left| \bigcup_{i \in [d], p \in \mP(r_i)}  \G_{(r_1, \ldots, r_i/p, \ldots, r_d)} \right|\\
={}& g(r_1, \ldots, r_d)-  \sum_{k \geq 1} (-1)^{k-1}  \left( \sum_{\substack{P_1 \subseteq \mP(r_1), \ldots, P_d \subseteq \mP(r_i) \\ \sum_{i=1}^d |P_i|= k}} \left|  \G_{(r_1/\pi(P_1), \ldots, (r_d/\pi(P_d))} \right|  \right) \\
={}& g(r_1, \ldots, r_d)+  \sum_{k \geq 1} (-1)^{k}  \left( \sum_{\substack{P_1 \subseteq \mP(r_1), \ldots, P_d \subseteq \mP(r_i) \\ \sum_{i=1}^d |P_i|= k}} g\left( \frac{r_1}{\pi(P_1)}, \ldots, \frac{r_d}{\pi(P_d)} \right)  \right) \\
={}&  \sum_{k \geq 0} (-1)^{k}  \left( \sum_{\substack{P_1 \subseteq \mP(r_1), \ldots, P_d \subseteq \mP(r_i) \\ \sum_{i=1}^d |P_i|= k}} g\left( \frac{r_1}{\pi(P_1)}, \ldots, \frac{r_d}{\pi(P_d)} \right)   \right) \\
={}& \sum \left( \prod_{i=1}^d \mu(q_i) \right) g\left( \frac{r_1}{q_1}, \ldots, \frac{r_d}{q_d} \right),
 \end{align*}
 where we made the substitution $q_i = \pi(P_i)$ in the last equality and applied~\eqref{Eq:0502}. For the same rationale as in the proof of $(i)$, the above summation can be taken over all $q_1,\ldots, q_d$ where $q_i | r_i$, and our claim follows.
\end{proof}

With Proposition~\ref{Prop:0501}, we obtain a recursive formula to compute the number of hypercube decompositions with a given lcm. In the case of $d=1$, we obtain the following sequences:
\[
\begin{array}{l|rrrrrrrrrrrrrrrrr}
n &  1 & 2 & 3 & 4 & 5 & 6 & 7 & 8 & 9 & 10& 11 & 12 & 13 & 14 & 15 & 16 & \cdots \\
\hline
g(n) &
1 & 2 & 2 & 5 & 2 & 12 & 2 & 26 & 9 & 36 &  2 & 206 & 2 &132& 40&677& \cdots \\
h(n) &
1 & 1 & 1 & 3 & 1 & 9 & 1 & 21 & 7 & 33 & 1 & 191 & 1 & 129 & 37 & 651 & \cdots 
 \end{array}
 \]
Moreover, given an NECS $C = \set{ \res{a_i}{n_i} : i \in [k]}$, if we define $\lcm(C)$ to be the least common multiple of $n_1, \ldots, n_k$, then it is not hard to adapt Lemma~\ref{Lem:0501} to show that $\Phi$ preserves lcm as well. For example, the last column in Figure~\ref{Fig:0501} gives the lcm of the corresponding decomposition $S$, as well as that of the NECS $\Phi(S)$. Thus, we see that $g(n)$ gives the number of NECS whose lcm divides $n$, and $h(n)$ gives the number of NECS whose lcm is exactly $n$. (Goulden et al.~\cite[Proposition 6]{GouldenGRS18} also obtained a recursive formula for finding the number of NECS with a given lcm, gcd, and number of residual classes.)

Also, recall that $\mu(ab) = \mu(a)\mu(b)$ whenever $a,b$ are coprime. Thus, if $r_1, \ldots, r_d \in \mN$ are pairwise coprime, then it follows from Proposition~\ref{Prop:0501} that $g(r_1, \ldots, r_d) = g\left( \prod_{i=1}^d r_i \right)$ and $h(r_1, \ldots, r_d) = h\left( \prod_{i=1}^d r_i \right)$. Thus, given $n \in \mN$ with prime factorization $n = p_1^{m_1} \cdots p_d^{m_d}$, there is a one-to-one correspondence between NECS with lcm $n$ and hypercube decompositions in $\S_d$ with lcm $(p_1^{m_1}, \ldots, p_d^{m_d})$. This geometric interpretation of NECS is similar to the ``low-dimensional mapping'' used by Berger, Felzenbaum, and Fraenkel, who obtained a series of results on covering systems by mapping them to lattice parallelotopes and employing geometric and combinatorial arguments. The reader can refer to~\cite{BergerFF87} and the references therein for their results. It would be interesting to investigate further what more can we learn about covering systems by relating them to hypercube decompositions.

\end{document}